\documentclass[a4paper, 12pt]{amsart}

\usepackage{amssymb,amsmath,latexsym,color,graphicx,stmaryrd,dsfont}
\usepackage{hyperref}

\newtheorem{dref}{Definition}[section] \newtheorem{lemma}[dref]{Lemma}
\newtheorem{theo}[dref]{Theorem} \newtheorem{prop}[dref]{Proposition}
\newtheorem{remark}[dref]{Remark} 

\usepackage{hyperref} 
\hypersetup{pdfborder=0 0 0, 
	    colorlinks=true,
	    citecolor=blue,
	    linkcolor=blue,
	    urlcolor=blue,
	    pdfauthor={Johannes Sjoestrand, Martin Vogel}
	   }
\usepackage[font=small,labelfont=bf]{caption}
\usepackage[margin=2.5cm]{geometry}
\newcommand{\defeq}{\stackrel{\mathrm{def}}{=}}
\newcommand{\C}{\mathbf{C}}

\newcommand{\N}{\mathbf{N}}
\newcommand{\Z}{\mathbf{Z}}
\newcommand{\spec}{\mathrm{Spec}}
\newcommand{\supp}{\mathop{\rm supp}}

\newcommand{\e}{\mathrm{e}}
\newcommand{\HS}{\mathrm{HS}}

\newcommand{\mO}{\mathcal{O}}
\title{General Toeplitz matrices subject to Gaussian perturbations}
\author{Johannes Sj\"ostrand}
\address[Johannes Sj\"ostrand]{IMB, 
  Universit\'e de Bourgogne Franche-Comt\'e, 
  UMR 5584 du CNRS, 
  9, avenue Alain Savary - BP 47870 - 21078 Dijon Cedex, France.}
\email{johannes.sjostrand@u-bourgogne.fr}
\author{Martin Vogel}
\address[Martin Vogel]{Institut de Recherche Math{\'e}matique Avanc{\'e}e, UMR 7501, Universit{\'e} de Strasbourg et CNRS, 7 rue Ren{\'e} Descartes, 67000 Strasbourg, France}
\email{vogel@math.unistra.fr}
 \date{}
\date{}
\begin{document}
\maketitle
\begin{abstract}
We study the spectra of general $N\times N$ Toeplitz matrices given by symbols in the 
Wiener Algebra perturbed by small complex Gaussian random matrices, in the regime $N\gg 1$. 
We prove an asymptotic formula for the number of eigenvalues of the perturbed 
matrix in smooth domains. We show that these eigenvalues follow a Weyl law with 
probability sub-exponentially close to $1$, as $N\gg1$, in particular that most eigenvalues of the 
 perturbed Toeplitz matrix are close to the curve in the complex plane given by the symbol of the 
 unperturbed Toeplitz  matrix. 
 \end{abstract}
 \setcounter{tocdepth}{1}
\tableofcontents
\section{Introduction and main result}\label{int}
\setcounter{equation}{0}
Let $a_{\nu}\in {\bf C}$, for $\nu\in {\bf Z}$ and assume that
\begin{equation}\label{unp.1}
|a_\nu |\le {\mathcal{O}}(1)m(\nu ),
\end{equation}
where $m:{\bf Z}\to ]0,+\infty [$ satisfies
\begin{equation}\label{unp.2}
	(1+|\nu|)m(\nu)\in \ell^1,
\end{equation}
and
\begin{equation}\label{unp.3}
	m(-\nu )=m(\nu ),\ \forall \nu \in \mathds{Z}.
\end{equation}
\par
Let
\begin{equation}\label{unp.5}
p(\tau )=\sum_{-\infty }^{+\infty }a_\nu \tau ^\nu ,
\end{equation}
act on complex valued functions on ${\bf Z}$. Here $\tau $ denotes
translation by 1 unit to the right: $\tau u(j)=u(j-1)$, $j\in {\bf
  Z}$. By (\ref{unp.2}) we know that $p(\tau )={\mathcal{O}}(1):\ell^2({\bf
  Z})\to \ell^2({\bf Z})$. Indeed, for the corresponding operator
norm, we have
\begin{equation}\label{unp.5.1}
\|p(\tau )\|\le \sum |a_j|\| \tau
^j\|=\|a\|_{\ell^1}\le {\mathcal{O}}(1)\|m\|_{\ell^1}.
\end{equation}
From the identity, $\tau (e^{ik\xi })=e^{-i\xi }e^{ik\xi }$, we
define the symbol of $p(\tau )$ by
\begin{equation}\label{unp.6}
  p(e^{-i\xi})=\sum_{-\infty }^\infty a_\nu e^{-i\nu \xi }.
\end{equation}
It is an element of the Wiener algebra \cite{BoSi99} and by \eqref{unp.2} in $C^1(S^1)$. 
\\
\par
We are interested in the \emph{Toeplitz matrix}
\begin{equation}\label{unp.7.5.1}
P_N \defeq 1_{[0,N[}p(\tau )1_{[0,N[},
\end{equation}
acting on $\C^N \simeq \ell^2([0,N[)$, for $1\ll N<\infty $. Furthermore, we frequently identify $\ell^2([0,N[)$ with the space
$\ell^2_{[0,N[}({\bf Z})$ of functions $u\in \ell^2({\bf Z})$ with support in $[0,N[$. 
\\
\par
The spectra of such Toeplitz matrices have been studied thoroughly, see \cite{BoSi99} for an overview. 
Let $P_\infty $ denote $p(\tau )$ as an operator $\ell^2({\bf Z})\to \ell^2({\bf Z})$. It is a normal operator and 
by Fourier series expansions, we see that the spectrum of $P_\infty $ is given by
\begin{equation}\label{unp.7}
	\sigma (P_\infty )=p(S^1).
\end{equation}
The restriction $P_{\N}=P_{\infty}|_{\ell^2(\N)}$ of $P_{\infty}$ to $\ell^2(\N)$, is in general no longer 
normal, except for specific choices of the coefficients $a_\nu$. The essential spectrum of the Toeplitz 
operator $P_{\N}$ is given by $p(S^1)$ and we have pointspectrum in all loops of $p(S^1)$ 
with non-zero winding number, i.e.  
\begin{equation}\label{unp7.1}
	\sigma(P_{\N}) = p(S^1) \cup \{ z\in \C; \mathrm{ind}_{p(S^1)}(z)\neq 0 \}.
\end{equation}
By a result of Krein \cite[Theorem 1.15]{BoSi99} the winding number of $p(S^1)$ 
around the point $z\not\in p(S^1)$ is related to the Fredholm index of $P_{\N}-z$: 
$\mathrm{Ind}(P_{\N}-z) = - \mathrm{ind}_{p(S^1)}(z)$.
\\
\par
The spectrum of the Toeplitz matrix $P_N$ is contained in a small neighborhood of the spectrum 
of $P_{\N}$. More precisely, for every $\epsilon >0$, 
\begin{equation}\label{unp7.2}
	\sigma(P_{N}) \subset \sigma(P_{\N})+D(0,\epsilon )
\end{equation}
for $N>0$ sufficiently large, where $D(z,r)$ denotes the open disc of
radius $r$, centered at $z$. Moreover, the limit of $\sigma(P_{N})$ as $N\to
\infty$ is contained in a union of analytic 
arcs inside $ \sigma(P_{\N})$, see \cite[Theorem 5.28]{BoSi99}. 
\\
\par
We show in Theorem \ref{main} below that after adding a small random perturbation to $P_N$, 
most of its eigenvalues will be close to the curve $p(S^1)$ with probability very close to $1$. 
See Figure \ref{fig1} below for a numerical illustration.
\subsection{Small Gaussian perturbation}
Consider the random matrix 
\begin{equation}\label{unp7.3}
 Q_{\omega}\defeq Q_{\omega}(N) \defeq  (q_{j,k}(\omega))_{1\leq j,k\leq N} 
\end{equation}
with complex Gaussian law 
\begin{equation*}
	(Q_{\omega})_*(d\mathds{P}) = \pi^{-N^2} \e^{-\|Q\|_{\mathrm{HS}}^2} L(dQ),
\end{equation*}
where $L$ denotes the Lebesgue measure on $\C^{N\times N}$. The 
entries $q_{j,k}$ of $Q_{\omega}$ are independent and identically 
distributed complex Gaussian random variables with expectation $0$, 
and variance $1$. 
\par 
We recall that the probability distribution of a complex Gaussian 
random variable $\alpha \sim \mathcal{N}_{\C}(0,1)$, is given by 
\begin{equation*}
	\alpha_*(d\mathds{P}) = \pi^{-1} \e^{-|\alpha|^2} L(d\alpha),
\end{equation*}
where $L(d\alpha)$ denotes the Lebesgue measure on $\C$. 
If $\mathds{E}$ denotes the expectation with respect to the probability 
measure $\mathds{P}$, then 
\begin{equation*}
	\mathds{E}[\alpha] = 0, \quad \mathds{E}[|\alpha|^2] = 1.
\end{equation*}

We are interested in studying the spectrum of the random perturbations of 
the matrix $P_N^0=P_N$: 
\begin{equation}\label{pert.1}
 P_N^{\delta} \defeq P_N^0 + \delta Q_{\omega}, 
 \quad 0 \leq \delta \ll 1.
\end{equation}
\subsection{Eigenvalue asympotics in smooth domains}
Let $\Omega \Subset \C$ be an open simply connected set with smooth 
boundary $\partial\Omega$, which is independent of $N$, satisfying 
\begin{enumerate}
	\item $\partial\Omega$ intersects $p(S^1)$ in at most finitely many 
		points;
	\item $p(S^1)$ does not self-intersect at these points of intersection;
	\item these points of intersection are non-critical, i.e. 
		\begin{equation*}%
			\partial_{\zeta}p \neq 0 \hbox{ on } p^{-1}(\partial\Omega \cap p(S^1) );
		\end{equation*}
         \item $\partial\Omega$ and $p(S^1)$ are transversal at every point of the 
         	intersection. 
\end{enumerate}
\begin{theo}\label{main}
  Let $p$ be as in \eqref{unp.6} and let $P_N^{\delta}$ be as in \eqref{pert.1}. 
  Let $\Omega$ be as above,  satisfying conditions (1) - (4),  pick a
  $\delta_0\in]0,1[$ and let $\delta_1 >3/2$. If  
\begin{equation}\label{m0}
	\e^{- N^{\delta_0} } \leq \delta \ll N^{-\delta_1}, 
\end{equation}
then there exists $\varepsilon_N = o(1)$, as $N\to \infty$, such that 
 \begin{equation}\label{m1}
	\left|\#(\sigma(P^{\delta}_N)\cap \Omega )
	-  \frac{N}{2\pi} \int_{S^1\cap\, p^{-1}(\Omega)}L_{S^1}(d\theta)\right| \leq \varepsilon_N N, 
\end{equation}
with probability 
 \begin{equation}\label{m2}
	 \geq 1 - \e^{-N^{\delta_0}}.
\end{equation}
\end{theo}
In \eqref{m1} we view $p$ as a map from $S^1$ to $\mathbf{C}$. 
Theorem \ref{main} shows that most eigenvalues of $P_N^{\delta}$ can be found close to 
the curve $p(S^1)$ with probability subexponentially close to $1$. This is illustrated in 
Figure \ref{fig1} for two different symbols. 
\begin{figure}[ht]
 \begin{minipage}[b]{0.49\linewidth}
  \centering
   \includegraphics[width=\textwidth]{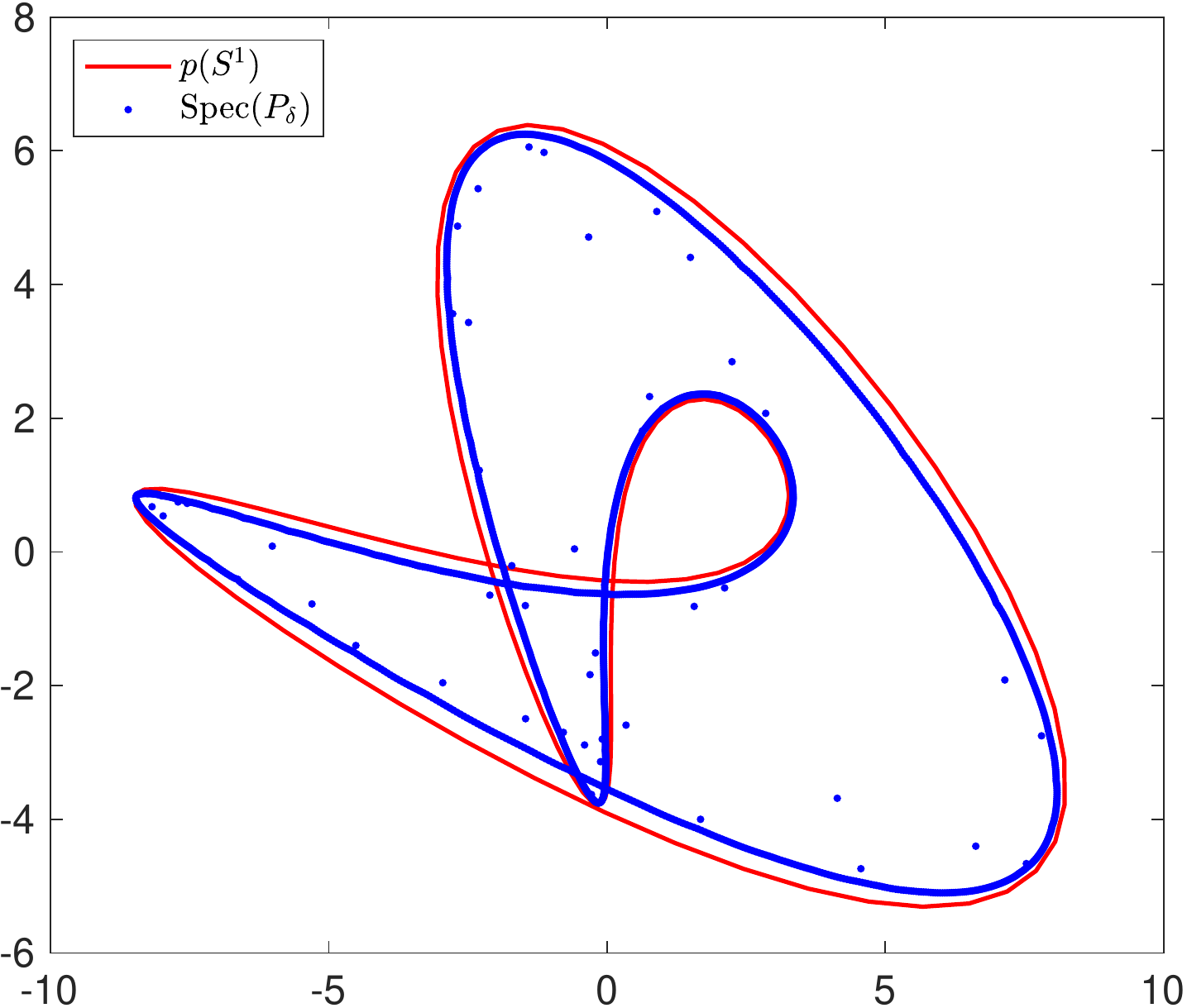}
 \end{minipage}
 \hspace{0cm}
 \begin{minipage}[b]{0.49\linewidth}
  \includegraphics[width=\textwidth]{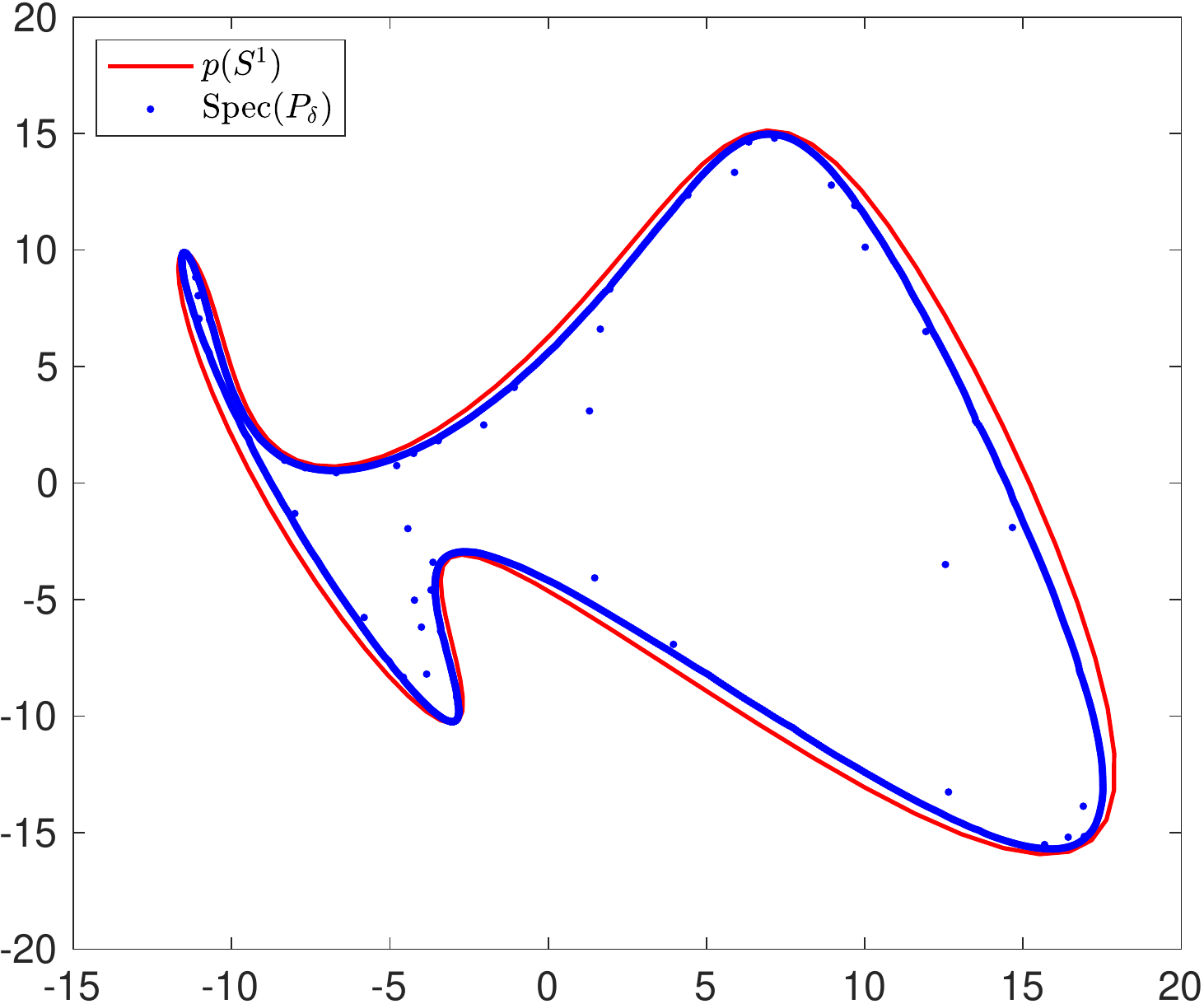}
 \end{minipage}
   \caption{The left hand side shows the spectrum of the perturbed Toeplitz matrix with symbol defined in \eqref{exp1}, \eqref{exp1.1} 
   	and the right hand side shows the spectrum of the perturbed Toeplitz matrix with symbol defined in \eqref{exp1.2}, \eqref{exp1.1} 
	The red line shows the symbol curve $p(S^1)$.}
	\label{fig1}
\end{figure}
The left hand side of Figure \ref{fig1} shows the spectrum of a perturbed Toeplitz matrix 
with $N=2000$ and $\delta =10^{-14}$, given by the symbol $p = p_0 + p_1$ where 
\begin{equation}\label{exp1}
	  p_0(1/\zeta) = -\zeta^{-4} -(3+2i)\zeta^{-3} +i\zeta^{-2}+\zeta^{-1} 
		 +10\zeta+(3+i)\zeta^2+4\zeta^3+i\zeta^4
\end{equation}
and 
\begin{equation}\label{exp1.1}
	  p_1(1/\zeta) = \sum_{\nu \in \Z} a_{\nu} \zeta^{\nu}, \quad a_{0}=0, ~
	  a_{-\nu} = 0.7|\nu|^{-5}+i|\nu|^{-9}
	  ,~ a_\nu = -2i\nu^{-5}+0.5\nu^{-9}~~ \nu \in \N.
\end{equation}
The red line shows the curve $p(S^1)$. The right hand side of Figure \ref{fig1} similarly shows the spectrum of the perturbed 
Toeplitz matrix given by $p= p_0 + p_1$ where $p_1$ is as above and 
\begin{equation}\label{exp1.2}
	    p_0(1/\zeta) = -4\zeta^{1}  -2i \zeta^{2}+ 2i\zeta^{-1}-\zeta^{-2}+2\zeta^{-3}.
\end{equation}
In our previous work \cite{SjVo19}, we studied Toeplitz matrices with a finite number of bands, given by 
symbols of the form 
\begin{equation}\label{fd.2}
p(\tau) = \sum_{j=-N_-}^{N_+} a_j \tau^j, \quad a_{-N_-},
a_{-N_-+1},\dots, a_{N_+} \in \C,\ a_{\pm N_\pm}\ne 0.
\end{equation}
In this case the symbols are analytic functions on $S^1$ and we are able to provide in  
\cite[Theorem 2.1]{SjVo19} a version of Theorem \ref{main} with a much sharper 
remainder estimate. See also \cite{Sj19,SjVo15b}, concerning the special cases of 
large Jordan block matrices $p(\tau) = \tau^{-1}$ and large bi-diagonal matrices
 $p(\tau) = a\tau + b\tau^{-1}$, $a,b\in \C$. 
However, Figure \ref{fig1} suggests that one could hope for a better remainder 
estimate in Theorem \ref{main} as well. 
\subsection{Convergence of the empirical measure and related results}
An alternative way to study the limiting distribution of the eigenvalues of $P_N^{\delta}$, up to 
errors of $o(N)$, is to study the \emph{empirical measure} of eigenvalues, defined by 
\begin{equation}\label{emp1}
	\xi_N \defeq \frac{1}{N}\sum_{\lambda \in \spec( P_N^{\delta})} \delta_{\lambda}
\end{equation}
where the eigenvalues are counted including multiplicity and $\delta_{\lambda}$ denotes 
the Dirac measure at $\lambda \in\C$. For any positive monotonically increasing function $\phi$  
on the positive reals and random variable $X$, Markov's inequality states that 
$\mathds{P} [ |X| \geq \varepsilon ] \leq \phi(\varepsilon)^{-1} \mathds{E}[ \phi(|X|)]$, assuming that the  
last quantity is finite. This yields that for $C_1>0$ large enough
\begin{equation}\label{mark1}
	\mathds{P} [ \| Q_{\omega}\|_{\HS} \leq \sqrt{C_1 N} ] \geq 1 - \e^{-N}.
\end{equation}
If $\delta \leq N^{-1}$, then \eqref{unp.5.1} and the the Borel-Cantelli Theorem shows that, 
almost surely, $\xi_N$ has compact support for $N>0$ sufficiently large. 
\par
We will show that, almost surely, $\xi_N$ converges weakly to the push-forward of the uniform 
measure on $S^1$ by the symbol $p$. 
\begin{theo}\label{main2}
	Let $\delta_0\in]0,1[$, let $\delta_1 >3/2$ and let $p$ be as in \eqref{unp.5}. If  
	\eqref{m0} holds, i.e. 
	\begin{equation*}
	\e^{- N^{\delta_0} } \leq \delta \ll N^{-\delta_1}
	\end{equation*}
	then, almost surely, 
	\begin{equation}\label{int.4}
		\xi_N \rightharpoonup  p_*\left(\frac{1}{2\pi} L_{S^1}\right), \quad N\to \infty, 
	\end{equation}
	weakly, where $L_{S^1}$ denotes the Lebesgue measure 
	on $S^1$.
\end{theo}
This result generalizes \cite[Corollary 2.2]{SjVo19} from the case of Toeplitz matrices with a 
finite number of bands to the general case \eqref{unp.5}. 
\\
\par
Similar results to Theorem \ref{main2} have been proven in various settings. 
In \cite{BPZ18,BPZ18b}, the authors consider the special case of band Toeplitz matrices, i.e. $P_N$ 
with $p$ as in \eqref{fd.2}. 
In this case they show that the convergence \eqref{int.4} holds weakly in probability for 
a coupling constant $\delta = N^{-\gamma}$, with $\gamma  >1/2$. Furthermore, 
they prove a version of this theorem for Toeplitz matrices with non-constant coefficients in 
the bands, see \cite[Theorem 1.3, Theorem 4.1]{BPZ18}. They follow a different approach than we 
do:  They compute directly the $\log |\det \mathcal{M}_N -z|$ by relating it to 
$\log |\det M_N(z)|$, where $M_N(z)$ is a truncation of $M_N -z$, where the smallest singular 
values of $M_N-z$ have been excluded. The level of truncation however depends on the strength 
of the coupling constant and it necessitates a very detailed analysis of the small singular values 
of $M_N -z$.
\\
\par
In the earlier work \cite{GuMaZe14}, the authors prove that the convergence 
\eqref{int.4} holds weakly in probability for the Jordan bloc matrix $P_N$ with $p(\tau) = \tau^{-1}$ \eqref{unp.5} and 
 a perturbation given by a complex Gaussian random matrix whose entries 
are independent complex Gaussian random variables whose variances vanishes (not necessarily at the same speed) polynomially fast, with minimal decay of order $N^{-1/2+}$. See also \cite{DaHa09} for a related result.
\par
In \cite{Wo16}, using a replacement principle developed in \cite{TVK10}, it was shown that the result of \cite{GuMaZe14} holds for perturbations given by 
complex random matrices whose entries are independent and identically distributed random 
complex random variables with expectation $0$ and variance $1$ and a coupling constant 
$\delta = N^{-\gamma}$, with $\gamma >2 $. 
\\
\\
\paragraph{\textbf{Acknowledgments}} The first author acknowledges support from the 2018 S. Bergman award.
The second author was supported by a CNRS Momentum fellowship. We are grateful to Ofer Zeitouni for his interest and 
a remark which lead to a better presentation of this paper. 
\section{The unperturbed operator}\label{unp}
\setcounter{equation}{0}
\par 
We are interested in the Toeplitz matrix
\begin{equation}\label{unp.7.5}
P_N=1_{[0,N[}p(\tau )1_{[0,N[}: \ell^2([0,N[)\to \ell^2([0,N[)
\end{equation}
for $1\ll N<\infty $, see also \eqref{unp.7.5.1}. Here we identify $\ell^2([0,N[)$ with the space
$\ell^2_{[0,N[}({\bf Z})$ of functions $u\in \ell^2({\bf Z})$ with
support in $[0,N[$. Sometimes we write $P_N=P_{[0,N[}$ and identify
$P_N$ with $P_{I}=1_Ip(\tau )1_I$ where $I=I_N$ is any interval in ${\bf Z}$ of
``length'' $|I|=\# I=N$.

\par Let $P_{\bf N}=P_{[0,+\infty [}$ and let $P_{{\bf
    Z}/\widetilde{N}{\bf Z}}$ denote $P=p(\tau )$, acting on
$\ell^2({\bf Z}/\widetilde{N}{\bf Z})$ which we identify with the space of
$\widetilde{N}$-periodic functions on ${\bf Z}$. Here
$\widetilde{N}\ge 1$. Using the the discrete Fourier transform,
we see that
\begin{equation}\label{unp.8}
\sigma (P_{{\bf Z}/\widetilde{N }{\bf
      Z}})=p(S_{\widetilde{N}}),
\end{equation}
where $S_{\widetilde{N}}$ is the dual of ${\bf Z}/\widetilde{N}\Z$ and given by
$$
S_{\widetilde{N}}=\{ e^{ik2\pi /\widetilde{N}};\, 0\le k<\widetilde{N}
\}.
$$
\par 
Let
\begin{equation}\label{unp.8.1}
p_N(\tau )=\sum_{|\nu |\le N}a_\nu \tau ^\nu=\sum_{\nu \in {\bf
    Z}}a_\nu ^N \tau ^\nu ,\ \ a_\nu ^N=1_{[-N,N]}(\nu )a_\nu . 
\end{equation}
and notice that
\begin{equation}\label{unp.9}
P_N=1_{[0,N[}\, p_N(\tau )1_{[0,N[}.
\end{equation}
We now consider $[0,N[$ as an interval $I_N$ in ${\bf Z}/\widetilde{N}{\bf
  Z}$, $\widetilde{N}=N+M$, where $M\in \{1,2,.. \}$ will be fixed and
independent of $N$. The matrix of $P_N$, indexed over $I_N\times
I_N$ is then given by
\begin{equation}\label{unp.10}
P_N(j,k)=a^N_{\widetilde{j}-\widetilde{k}},\ j,k\in I_N\subset {\bf Z}/\widetilde{N}Z,
\end{equation}
where $\widetilde{j},\widetilde{k}\in {\bf Z}$ are the preimages of
$j,k$ under the projection ${\bf Z}\to{\bf Z} /\widetilde{N}{\bf Z}$, that belong
to the interval $[0,N[\subset {\bf Z}$. 
\par
Let $\widetilde{P}_N$
 be given by the formula (\ref{unp.9}), with the difference that we
 now view $\tau $ as a translation on $\ell^2({\bf
   Z}/\widetilde{N}{\bf Z})$:
\begin{equation}\label{unp.10.1}
\widetilde{P}_N=1_{I_N}p_N(\tau )1_{I_N}.
\end{equation}
The matrix of $\widetilde{P}_N$ is given by
\begin{equation}\label{unp.11}
  \widetilde{P}_N(j,k)=\sum_{\nu \in {\bf Z},\atop
  \nu \equiv j-k\, \mathrm{mod\,}\widetilde{N}{\bf Z}}a_\nu ^N,\quad
j,k\in I_N.
\end{equation}

Alternatively, if we let $\widetilde{j},\widetilde{k}$ be the
preimages in $[0,N[$ of $j,k\in I_N$, then
\begin{equation}\label{unp.12}
  \widetilde{P}_N(j,k)=\sum_{\widehat{j}\in{\bf Z};\ \widehat{j}\equiv
    \widetilde{j}\,\mathrm{mod\,}\widetilde{N}{\bf Z}}a^N_{\widehat{j}-\widetilde{k}}.
\end{equation}
Recall that the terms in (\ref{unp.11}), (\ref{unp.12}) with $|\nu |>N$
or $|\widehat{j}-\widetilde{k}|>N$ do vanish. This implies that with
$\widetilde{j}$, $\widetilde{k}$ as in (\ref{unp.12}),
\begin{equation}\label{unp.13}
\widetilde{P}_N(j,k)-P_N(j,k)=a^N_{\widetilde{j}-\widetilde{N}-\widetilde{k}}+a^N_{\widetilde{j}+\widetilde{N}-\widetilde{k}}.
\end{equation}
Here
\[
  \begin{split}
&\widetilde{j}-\widetilde{N}\in [0,N[-\widetilde{N}=[-\widetilde{N},N-\widetilde{N}[=[-N-M,-M[,\\
&
\widetilde{j}+\widetilde{N}\in [0,N[+\widetilde{N}=[\widetilde{N},N+\widetilde{N}[=[N+M,2N+M[.
\end{split}
\]
Since $\widetilde{k}\in [0,N[$ we have for the first term in
(\ref{unp.13}) that
$|\widetilde{j}-\widetilde{N}-\widetilde{k}|=\widetilde{k}+M+(N-\widetilde{j})$
with nonnegative terms in the last sum. Similarly for the second term
in (\ref{unp.13}), we have
$|\widetilde{j}+\widetilde{N}-\widetilde{k}|=\widetilde{j}+M+(N-\widetilde{k})$
where the terms in the last sum are all $\ge 0$.

\par It follows that the trace class norm of $P_N-\widetilde{P}_N$ is
bounded from above by
\begin{equation*}
\begin{split}
\sum_{j<-M,\ k\ge 0}|a_{j-k}| &+\sum_{j\ge N+M,\ k<N }|a_{j-k}|
=\sum_{k\ge 0,\ j\le -M}|a_{j-k}|+\sum_{k\le 0,\ j\ge M}|a_{j-k}|\\
&\le 2C\sum_{k=0}^\infty \sum_{j=0}^\infty m(M+k+j)
=2C \sum_{k=0}^\infty (k+1)m(M+k)\\
&=2C\sum_{k=M}^\infty (k+1-M)m(k).
\end{split}
\end{equation*}
By (\ref{unp.2}), it follows that 
\begin{equation}\label{unp.15}
  \|
P_N-\widetilde{P}_N\|_{\mathrm{tr}}\le 2C\sum_{k=M}^{+\infty
}(k+1-M)m(k)\to 0,\ M\to \infty ,
\end{equation}
uniformly with respect to $N$.

\section{A Grushin problem for $P_N-z$}\label{gr} 
\setcounter{equation}{0}
Let $K\Subset \C$ be an open relatively compact set and let $z\in K$. 
Consider 
\begin{equation}\label{gr.11}
J=[-M,0[,\ I_N=[0,N[
\end{equation}
as subsets of ${\bf Z}/(N+M){\bf Z}$ so that
$$
	J \cup I_N={\bf Z}/(N+M){\bf Z}=:{\bf Z}_{N+M}
$$
is a partition. 
Recall \eqref{unp.8.1}, \eqref{unp.10.1} and consider
$$
p_N(\tau )-z:\ell^2({\bf Z}_{N+M})\to \ell^2({\bf Z}_{N+M})
$$
and write this operator as a $2\times 2$ matrix
\begin{equation}\label{gr.2}
p_N-z=\begin{pmatrix} \widetilde{P}_N-z &R_-\\ R_+& R_{+-}(z)\end{pmatrix},
\end{equation}
induced by the orthogonal decomposition
\begin{equation}\label{gr.1}
\ell^2({\bf Z}_{N+M})=\ell^2(I_N)\oplus \ell^2(J).
\end{equation}
\par 
The operator $p_N(\tau )$ is normal and we know by 
 \eqref{unp.8} that its spectrum is
\begin{equation}\label{gr.3}
\sigma (p_N(\tau ))=p_N(S_{N+M}).
\end{equation}
Replacing $\widetilde{P}_N$ in \eqref{gr.2} by $P_N$ \eqref{unp.9}, we put 
\begin{equation}\label{gr.4}
{\mathcal{P}}_N(z)=\begin{pmatrix}P_N-z & R_-\\ R_+ &R_{+-}(z)\end{pmatrix}.
\end{equation}
Then, by (\ref{unp.15}),
\begin{equation}\label{gr.5}
  \|
{\mathcal{P}}_N(z)-(p_N-z)\|_{\mathrm{tr}}\le 2C\sum_{k=M}^{+\infty
}(k+1-M)m(k)=:\epsilon (M) .
\end{equation}
If $\epsilon (M)<\mathrm{dist\,}(z,p_N(S_{N+M}))=:d_N(z)$, then ${\mathcal{P}
  }(z)$ is bijective and
\begin{equation}\label{gr.6}
\| {\mathcal{P}}_N(z)^{-1}\|\le \frac{1}{d_N(z)-\epsilon (M)}.
\end{equation}
Write,
\begin{equation*}
\begin{split}
  {\mathcal{P}}_N(z) &=p_N(\tau )-z+{\mathcal{P}}_N(z)-(p_N(\tau )-z)\\
&=(p_N(\tau )-z)\left(1+(p_N(\tau )-z)^{-1}({\mathcal{P}}_N(z)-(p_N(\tau) -z)) \right).
\end{split}
\end{equation*}
Here,
\begin{equation*}
\begin{split}
  \big|
\det \big(1+(p_N(\tau )-z)^{-1}&({\mathcal{P}}_N(z)-(p_N(\tau )-z))\big)
\big|\\
&\le \exp \| (p_N(\tau )-z)^{-1}({\mathcal{P}}_N(z)-(p_N(\tau
)-z))\|_{\mathrm{tr}}\\
&\le \exp (\epsilon (M)/d_N(z)),
\end{split}
\end{equation*}
so 
\begin{equation}\label{gr.7}
|\det {\mathcal{P}}_N(z)|\le | \det (p_N(\tau )-z) |\, e^{\epsilon (M)/d_N(z)}.
\end{equation}
Similarly from
\begin{equation*}
\begin{split}
  p_N(\tau )-z &={\mathcal{P}}_N(z)+p_N(\tau )-z-{\mathcal{P}}_N(z)\\
  &={\mathcal{P}}_N(z)\left( 1+{\mathcal{P}}_N(z)^{-1}(p_N(\tau )-z-{\mathcal{P}}_N(z) \right),
\end{split}
\end{equation*}
we get
\begin{equation}\label{gr.8}
  |\det (p_N(\tau )-z)|\le |\det {\mathcal{P}}_N(z)| e^{\frac{\epsilon
      (M)}{d_N(z)-\epsilon (M)}}.
\end{equation}

\par In analogy with (\ref{gr.4}), we write
\begin{equation}\label{gr.10}
  {\mathcal{P}}_N(z)^{-1}={\mathcal{E}}_N(z)=\begin{pmatrix}E^N &E_+^N\\
E_-^N &E_{-+}^N
\end{pmatrix}
:\ \ell^2(I_N)\oplus \ell^2(J)\to \ell^2(I_N)\oplus \ell^2(J),
\end{equation}
where $J$, $I_N$ were defined in (\ref{gr.11}),
still viewed as intervals in ${\bf Z}_{N+M}$. From (\ref{gr.6}) we get
for the respective operator norms:
\begin{equation}\label{gr.12}
\| E^N\|, \| E^N_+\|, \| E^N_-\|, \| E^N_{-+}\|\le (d_N(z)-\epsilon (M))^{-1}.
\end{equation}

\section{Second Grushin problem}\label{2gr}
\setcounter{equation}{0}
We begin with a result, which is a generalization of 
\cite[Proposition 3.4]{SjZw07} to the case where $R_{+-}\neq 0$. 
\begin{prop}\label{prop:CGP}
Let ${\mathcal{H}}_1, {\mathcal{H}}_2, {\mathcal{H}}_{\pm}, {\mathcal{S}}_{\pm}$ be Banach spaces. 
If 
\begin{equation}\label{2gr.1}
  {\mathcal{P}}=\begin{pmatrix}P &R_-\\ R_+ &R_{+-}\end{pmatrix}:
{\mathcal{H}}_1\times {\mathcal{H}}_-\to  {\mathcal{H}}_2\times  {\mathcal{H}}_+  
\end{equation}
is bijective with bounded inverse,
$$
{\mathcal{E}}=\begin{pmatrix}E &E_+\\ E_- &E_{-+}\end{pmatrix}:
{\mathcal{H}}_2\times {\mathcal{H}}_+\to  {\mathcal{H}}_1\times  {\mathcal{H}}_-, 
$$
and if
\begin{equation}\label{2gr.2}
  {\mathcal{S}}=\begin{pmatrix}E_{-+} &S_-\\ S_+ & 0\end{pmatrix}:
  {\mathcal{H}}_+\times {\mathcal{S}}_-\to {\mathcal{H}}_-\times {\mathcal{S}}_+
\end{equation}
is bijective with bounded inverse
$$
{\mathcal{F}}=\begin{pmatrix}F &F_+\\ F_- &F_{-+}\end{pmatrix}:
{\mathcal{H}}_-\times {\mathcal{S}}_+\to  {\mathcal{H}}_+\times {\mathcal{S}}_-, 
$$
then
\begin{equation}\label{2gr.3}
  {\mathcal{T}}=\begin{pmatrix} P &R_-S_-\\
S_+R_+ &S_+R_{+-}S_-
\end{pmatrix}=:
\begin{pmatrix}T &T_-\\ T_+ &T_{+-}\end{pmatrix}
:
{\mathcal{H}}_1\times {\mathcal{S}}_-\to {\mathcal{H}}_2\times {\mathcal{S}}_+
\end{equation}
is bijective with the bounded inverse
\begin{equation}\label{2gr.4}
  {\mathcal{G}}=\begin{pmatrix}G &G_+\\ G_- &G_{-+}\end{pmatrix}=
  \begin{pmatrix}E-E_+FE_- &E_+F_+\\ F_-E_- &-F_{-+}\end{pmatrix}
  :{\mathcal{H}}_2\times {\mathcal{S}}_+ \to {\mathcal{H}}_1\times {\mathcal{S}}_-
\end{equation}
\end{prop}
\begin{proof}
	We can essentially follow the proof of \cite[Proposition 3.4]{SjZw07}. We need 
	to solve 
	\begin{equation}\label{2gr.prop.1}
		\begin{cases}
			Pu+R_-S_-u_- = v \\
			S_+R_+u + S_+R_{+-}S_-u_-=v_+.
		\end{cases}
	\end{equation}
	Putting $\widetilde{v}_+=R_+u+R_{+-}S_-u_-$, the first equation 
	is equivalent to 
	\begin{equation*}
		\begin{cases}
			Pu+R_-S_-u_- = v \\
			R_+u+R_{+-}S_-u_-=\widetilde{v}_+, 
		\end{cases}
		\quad \text{i.e.} \quad 
		\mathcal{P} \begin{pmatrix} u \\ S_-u_- \\ \end{pmatrix} =
		\begin{pmatrix} v \\ \widetilde{v}_+\\ \end{pmatrix},
	\end{equation*}
	and hence to 
	\begin{equation}\label{2gr.prop.2}
		\begin{cases}
			u = Ev+E_+\widetilde{v}_+ \\
			S_-u_-=E_-v + E_{-+}\widetilde{v}_+.
		\end{cases}
	\end{equation}
	Therefore, we can replace $u$ by $\widetilde{v}_+$ and \eqref{2gr.prop.1} 
	is equivalent to 
	\begin{equation}\label{2gr.prop.3}
		\begin{pmatrix} 	E_{-+} & S_- \\ S_+ & 0 \\	\end{pmatrix} 
		\begin{pmatrix} \widetilde{v}_+ \\ -u_- \\ \end{pmatrix} =
		\begin{pmatrix} -E_-v \\ v_+ \\ \end{pmatrix}
	\end{equation}
	which can be solved by $\mathcal{F}$. Hence, \eqref{2gr.prop.3} is equivalent to 
	\begin{equation*}
		\begin{cases}
			\widetilde{v}_+ = - FE_-v + F_+v_+ \\ 
			-u_-= - F_-E_-v+F_{-+}v_+,
		\end{cases}
	\end{equation*}
	and \eqref{2gr.prop.2} gives the unique solution of \eqref{2gr.prop.1} 
	\begin{equation*}
		\begin{cases}
			u = (E - E_+FE_-)v + E_+F_+v_+ \\ 
			u_- = F_-E_-v - F_{-+}v_+. 
		\end{cases}
		\qedhere
	\end{equation*}
\end{proof}

\par 
\subsection{Grushin problem for $E_{-+}(z)$}\label{sec:RPG}
We want to apply Proposition \ref{prop:CGP} to ${\mathcal{P}}={\mathcal{P}}(z)={\mathcal{P}}_N(z)$ in
(\ref{gr.4}) with the inverse ${\mathcal{E}}={\mathcal{E}}_N(z)$ in
(\ref{gr.10}), where we sometimes drop the index $N$. We begin by constructing an 
invertible Grushin problem for $E_{-+}$:
\\
\par
Let $0\leq t_1 \leq \dots \leq t_M$ denote the singular values of $E_{-+}(z)$. Let $e_1, \dots, e_M$ denote 
an orthonormal basis of eigenvectors of $E_{-+}^*E_{-+}$ associated to the eigenvalues $t_1^2 \leq \dots \leq t_M^2$. 
Since $E_{-+}$ is Fredholm of index $0$, we have that $\dim \mathcal{N}(E_{-+}(z)) = \dim \mathcal{N}(E_{-+}^*(z))$. 
Using the spectral decomposition $\ell^2(J) = \mathcal{N}(E_{-+}^*E_{-+}) \oplus_{\perp} \mathcal{R}(E_{-+}^*E_{-+})$ together with the fact that $\mathcal{N}(E_{-+}^*E_{-+}) = \mathcal{N}(E_{-+}) $ and 
$\mathcal{R}(E_{-+}^*) = \mathcal{N}(E_{-+})^{\perp}$, it follows that 
$\mathcal{R}(E_{-+}^*) = \mathcal{R}(E_{-+}^*E_{-+})$. Similarly, we get that 
$\mathcal{R}(E_{-+}) = \mathcal{R}(E_{-+}E_{-+})^*$. One then easily checks that $E_{-+}: \mathcal{R}(E_{-+}^*E_{-+}) \to \mathcal{R}(E_{-+}E_{-+}^*)$ is a bijection. Similarly, we get that $E_{-+}^*: \mathcal{R}(E_{-+}E_{-+}^*) \to \mathcal{R}(E_{-+}^*E_{-+})$  
is a bijection. Let $f_1,\dots,f_{M_0}$ denote an orthonormal basis of  $\mathcal{N}(E_{-+}^*(z))$ and set 
\begin{equation*}
	f_j = t_j^{-1} E_{-+} e_j , \quad j=M_0+1, \dots, M.
\end{equation*}
Then, $f_1,\dots,f_M$ is an orthonormal basis of $\ell^2(J)$ comprised of eigenfunctions of $E_{-+}E_{-+}^*$ associated with 
the eigenvalues $t_1^2 \leq \dots \leq t_M^2$. In particular, $\sigma(E_{-+}E_{-+}^*) =  \sigma(E_{-+}^*E_{-+})$ and 
\begin{equation}\label{2gr.5.0}
	 E_{-+} e_j = t_j f_j, \quad E_{-+}^* f_j = t_je_j, \quad j=1, \dots , M.
\end{equation}
\par
Let $0\le t_1\le ...\le t_k$ be the singular values of $E_{-+}(z)$ in the
interval $[0,\tau ]$ for $\tau >0$ small. 
Let ${\mathcal{S}}_+,\, {\mathcal{S}}_-\subset \ell^2(J)$
be the corresponding (sums of) spectral subspaces for
$E_{-+}^*E_{-+}$ and $E_{-+}E_{-+}^*$ respectively, corresponding to
the eigenvalues $t_1^2\le t_2^2\le ... \le t_k^2$ in $[0,\tau
^2]$. Using \eqref{2gr.5.0}, we see that the restrictions (denoted by the same symbols)
$$
E_{-+}:{\mathcal{S}}_+\to {\mathcal{S}}_-,\ E_{-+}^*:{\mathcal{S}}_-\to {\mathcal{S}}_+, 
$$
have norms $\le \tau $. Also,
\begin{equation}\label{2gr.5.2}
E_{-+}:{\mathcal{S}}_+^\perp \to {\mathcal{S}}_-^\perp,\ E_{-+}^*: {\mathcal{S}
  }_-^\perp\to  {\mathcal{S}}_+^\perp  
\end{equation}
are bijective with inverses of norm $\le 1/\tau $.

Let $S_+$ be the orthogonal projection onto
${\mathcal{S}}_+$, viewed as an operator $\ell^2(J)\to {\mathcal{S}}_+$, whose
adjoint is the inclusion map ${\mathcal{S}}_+\to \ell^2(J)$. Let $S_-:{\mathcal{S}
  }_-\to \ell^2(J)$ be the inclusion map. Let ${\mathcal{S}}$ be the
operator in (\ref{2gr.2}) with ${\mathcal{H}}_\pm=\ell^2(J)$, corresponding
to the problem
\begin{equation}\label{2gr.5}
  \begin{cases}
    E_{-+}g+S_-g_-=h\in \ell^2(J),\\
    S_+g=h_+\in {\mathcal{S}}_+,
  \end{cases}
\end{equation}
for the unknowns $g\in \ell^2(J)$, $g_-\in {\mathcal{S}}_-$. Using the
orthogonal decompositions,
$$
\ell^2(J)={\mathcal{S}}_+^\perp\oplus {\mathcal{S}}_+,\ \ell^2(J)={\mathcal{S}
  }_-^\perp\oplus {\mathcal{S}}_-,
$$
we write $g=\sum_1^kg_je_j + g^{\perp}$ and $h=\sum_1^kh_jf_j + h^{\perp}$. Then, \eqref{2gr.5} 
is equivalent to 
\begin{equation*}
  \begin{cases}
    g^{\perp}=(E_{-+})^{-1}h^{\perp}\\
    \begin{pmatrix}g_j\\ g_-^j\end{pmatrix} =
\begin{pmatrix}0 &1\\ 1 & - t_j\end{pmatrix}
\begin{pmatrix}h_j \\ h_+^j\end{pmatrix}, ~~ j=1,\dots,M,
  \end{cases}
\end{equation*}
where we also used that $g_-=\sum_1^kg_-^jf_j$ and $h_+=\sum_1^kh_+^je_j$. 
It follows that
\begin{equation}\label{2gr.5.1}
  \begin{cases}
    g = (E_{-+})^{-1}h^{\perp} + \sum_1^k h_+^je_j \\ 
    g_- =  \sum_1^k h^j f_j - \sum_1^k t_j h_+^jf_j.
  \end{cases}
\end{equation}
Hence, the unique solution to \eqref{2gr.5} is given by 
\begin{equation}\label{2gr.6}
\begin{pmatrix}g\\ g_-\end{pmatrix}={\mathcal{F}}\begin{pmatrix}h\\
  h_+\end{pmatrix} =
\begin{pmatrix}F &F_+\\ F_- &F_{-+}\end{pmatrix}
\begin{pmatrix}h \\ h_+\end{pmatrix},
\end{equation}
where
\begin{equation}\label{2gr.7}
  \begin{split}
    &F=E_{-+}^{-1}\Pi _{{\mathcal{S}}_-^\perp},\ \ F_+=S_+^*,\\
    &F_-=S_-^*,\ \ F_{-+}=-{{E_{-+}}_\vert}_{{\mathcal{S}}_+}:\, {\mathcal{S}}_+\to {\mathcal{S}}_- .
\end{split}    
\end{equation}
Here $\Pi_{B}$ denotes the orthogonal projection onto the subspace $B$
of $A$, viewed as a self-adjoint operator $A\to A$. Notice that $F = \Pi_{\mathcal{S}_+^{\perp}} F$ and 
that 
\begin{equation}\label{2gr.7.1}
	F_{-+} = - \sum_1^k t_j f_j \circ e_j^*.
\end{equation}
Using as well \eqref{2gr.5.2}, we have
\begin{equation}\label{2gr.8}
\|F\|\le 1/\tau ,\ \|F_+\|, \|F_-\|\le 1, \|F_{-+}\|\le \tau .
\end{equation}
\subsection{Composing the Grushin problems}
From now on we assume that
\begin{equation}\label{det.1}
0<\alpha \ll 1,\ \ \epsilon (M)\le \alpha /2,
\end{equation}
and the estimates below will be uniformly valid for $z\in K\setminus
\gamma _\alpha $, $N\gg 1$, where $K$ is some fixed relatively compact open set in
${\bf C}$ and
\begin{equation}\label{det.2}
\gamma _\alpha =\{ z\in {\bf C};\, \mathrm{dist\,}(z,\gamma )\le\alpha
\},\ \ \gamma =p(S^1).
\end{equation}
We apply Proposition \ref{prop:CGP} to ${\mathcal{P}_N}$ in
(\ref{gr.4}) with the inverse $\mathcal{E}_N$ in \eqref{gr.10}, and to 
$\mathcal{S}$ defined in \eqref{2gr.5} with inverse in $\mathcal{F}$ in \eqref{2gr.6}.
Let $z\in K\backslash\gamma_\alpha$, then  
\begin{equation}\label{2gr.10.0}
  {\mathcal{T}_N}=\begin{pmatrix} P_N-z &R_-S_-\\
S_+R_+ &S_+R_{+-}S_-
\end{pmatrix} =
\begin{pmatrix}T &T_-\\ T_+ &T_{+-}\end{pmatrix}
:
L^2(I_N)\times {\mathcal{S}}_-\to L^2(I_N)\times {\mathcal{S}}_+,
\end{equation}
defined as in (\ref{2gr.3}), is bijective with the bounded inverse
\begin{equation}\label{2gr.10.1}
  {\mathcal{G}_N}=\begin{pmatrix}G^N &G_+^N\\ G_-^N &G_{-+}^N\end{pmatrix}=
  \begin{pmatrix}E^N-E_+^NFE_-^N &E_+^NF_+\\ F_-E_-^N &-F_{-+}\end{pmatrix}.
\end{equation}
Since $S_\pm$ have norms $\le 1$, we get
\begin{equation}\label{2gr.9}
\| T_\pm\| \le \| R_\pm\| = \mO(1),
\end{equation}
uniformly in $N$, $\alpha$ and $z\in K$. 
Also, since the norms of $E^N,E_+^N, E_-^N$ are $\leq 2/\alpha$ (uniformly as $N\to \infty$) by \eqref{gr.12}, 
we get from (\ref{2gr.4}), (\ref{2gr.8}), that
\begin{equation}\label{2gr.10}
\| G^N\|\le \frac{2}{\alpha}+\frac{4}{\tau \alpha^2},\  \|G^N_{-+}\|\le \tau ,\
\|G^N_\pm\|\le \frac{2}{\alpha}.
\end{equation}
\begin{prop}\label{prop:gr1}
  Let $K\Subset \C$ be an open relatively compact set, let $z\in K\backslash \gamma_\alpha$, 
  and let $\tau>0$ be as in the definition of the Grushin problem \eqref{2gr.5}.  
  Then, for $\tau>0$ small enough, depending only on $K$, we have that $G_+$ is 
  injective and $G_-$ is surjective. Moreover, there exists a constant $C>0$, depending only on $K$, 
  such  that for all $z\in K\backslash \gamma_{\alpha}$ the singular values $s_j^+$ of $G_+$, 
  and $s_j^-$ of $G_-^*$ satisfy 
 \begin{equation}\label{2gr.10.2}
\frac{1}{C}\le s_j^\pm \le \frac{2}{\alpha}, \quad 1 \leq j \leq k(z) =\mathrm{rank}(G_\pm).
\end{equation}
\end{prop}
\begin{proof}
We begin with the injectivity of $G_+$. From
\begin{equation}\label{2gr.11}
  \begin{pmatrix}P &T_-\\ T_+ &T_{+-}\end{pmatrix}
  \begin{pmatrix}G &G_+\\ G_- &G_{-+}\end{pmatrix}=1,
\end{equation}
we have $T_+G_++T_{+-}G_{-+}=1$ which we write
$T_+G_+=1-T_{+-}G_{-+}$. Here
$$
\|T_{+-}G_{-+}\|\le \| R_{+-}\|\tau ={\mathcal{O}}(\tau ),
$$
where we used that $\| R_{-+} \| \leq \|p(\tau)-z\| = \mO(1)\| m\|_{\ell^1}$, thus the error term above 
only depends on $K$. 
Choosing $\tau >0$ small enough, depending on $K$ but not on $N$, we get that 
$\|T_{+-}G_{+-}\|\le 1/2$. Then $1-T_{+-}G_{-+}$ is bijective with $\|
(1-T_{+-}G_{-+})^{-1}\|\le 2$ and $G_+$ {\it has the left inverse}
\begin{equation}\label{2gr.12}
(1-T_{+-}G_{-+})^{-1}T_+
\end{equation}
{\it of norm} $\le 2\| R_+\|={\mathcal{O}}(1)$, depending only on $K$. 
\\
\par
Now we turn to the surjectivity of $G_-$. From
$$
\begin{pmatrix}G & G_+\\ G_-& G_{-+}\end{pmatrix}
\begin{pmatrix}P-z &T_-\\ T_+ &T_{+-}\end{pmatrix}= 1,
$$
we get
$$
\begin{pmatrix}(P-z)^* &T_+^*\\ T_-^* &T_{+-}^*\end{pmatrix}
\begin{pmatrix}G^* & G_-^*\\ G_+^*& G_{-+}^*\end{pmatrix}
= 1,
$$
and as above we then see that $G_-^*$ has the left inverse
$(1-T_{+-}^*G_{-+}^*)^{-1}T_-^*$. Hence $G_-$ {\it has  the right inverse}
\begin{equation}\label{2gr.13}
T_-(1-G_{-+}T_{+-})^{-1},
\end{equation}
of norm $\leq 2\| R_-\| = \mO(1)$, depending only on $K$. 
\\
\par
The lower bound on the singular values follows from the estimates on the left inverses of $G_+$ and $G_-^*$, 
and the upper bound follows from \eqref{2gr.10}.
\end{proof}
\section{Determinants}\label{det}
\setcounter{equation}{0}

We continue working under the assumptions \eqref{det.1}, \eqref{det.2}. 
Additionally, we fix $\tau>0$ sufficiently small (depending only 
on the fixed relatively compact set $K\Subset \C$) so that $\| T_{+-}G_{-+}\|$,
$\|G_{-+}T_{+-}\|$ (both $={\mathcal{O}}(\tau )$) are $\le 1/2$, which 
implies that $G_+$ is injective and $G_+$ is surjective, see Proposition \ref{prop:gr1}. 
\par
From now on we will work with $z\in K\backslash \gamma_{\alpha}$. 
The constructions and estimates in Section \ref{gr} are then uniform
in $z$ for $N\gg 1$ and the same holds for those in Section \ref{2gr}. 
\begin{remark}
	To get the $o(N)$ error term in Theorem \ref{main}, we will take 
	$\alpha >0 $ arbitrarily small, and $M>1$ large enough (but fixed) so that 
	$\varepsilon(M) \leq \alpha/2$, see \eqref{unp.15} as well as $N>1$ sufficiently large. 
	In the following, the error terms will typically depend on $\alpha$, although we will not always denote 
	this explicitly, however they will be uniform in $N>1$ and in $z\in K\backslash\gamma_{\alpha}$. 
\end{remark}
\par
\subsection{The unperturbed operator}
For $z\in K\setminus \gamma _\alpha $, we have $d_N(z)\ge \alpha
$ and (\ref{gr.7}), (\ref{gr.8}) give
\begin{equation}\label{det.3}
| \det {\mathcal{P}}_N(z)|\le e^{\epsilon (M)/\alpha }|\det (p_N(\tau
)-z)| ,
\end{equation}
\begin{equation}\label{det.4}
 |\det (p_N(\tau )-z)|
\le e^{2\epsilon (M)/\alpha }
 | \det {\mathcal{P}}_N(z)|,
\end{equation}
where we also used that
$$
\frac{\epsilon (M)}{d_N(z)-\epsilon (M)}\le \frac{\epsilon (M)}{\alpha
-\epsilon (M)}\le \frac{2\epsilon (M)}{\alpha },
$$
by the second inequality in (\ref{det.1}).  Recall here that
$p_N(\tau )$ acts on $\ell^2({\bf Z}/\widetilde{N}{\bf Z})$,
$\widetilde{N}=N+M$.

\par 
By the Schur complement formula, we have
\begin{equation}\label{det.5}
  \begin{split}
  \det (P_N-z) = \det {\mathcal{P}}(z)\, \det E_{-+}(z),\\
\det (P_N-z) = \det {\mathcal{T}}(z)\, \det G_{-+}(z),
\end{split}
\end{equation}
so
\begin{equation}\label{det.6}
\frac{\det {\mathcal{T}}}{\det {\mathcal{P}}}=\frac{\det E_{-+}}{\det G_{-+}}.
\end{equation}
Recall from the Section \ref{sec:RPG} that the singular values of $E_{-+}$
are denoted by $0\le t_1\le t_2\le \dots \leq t_M$ and that those of $G_{-+}$
are $t_1,...,t_k$, where $k=k(z,N)$ is determined by the condition
$t_k\le \tau <t_{k+1}$. Thus
$$
\left| \frac{\det E_{-+}}{\det G_{-+}} \right|=\prod_{k+1}^M t_j$$
and we get (since $\tau \ll 1$)
$$
\tau ^M\le \left| \frac{\det E_{-+}}{\det G_{-+}} \right| \le \left( \frac{2}{\alpha} \right)^M.
$$
Since $\tau >0$ is small, but fixed depending only on $K$, we have uniformly for 
$z\in K\setminus \gamma _\alpha $, $N\gg 1$:
\begin{equation}\label{det.7}
\left| \ln |\det E_{-+}|-\ln |\det G_{-+}|\right|\le \mO(1) %
\end{equation}
and by (\ref{det.6})
\begin{equation}\label{det.8}
\left| \ln |\det {\mathcal{T}}|-\ln |\det {\mathcal{P}}|\right|\le \mO(1).% 
\end{equation}
From (\ref{det.3}), (\ref{det.4}), we get
\begin{equation}\label{det.9}
\left| \ln |\det {\mathcal{P}}|-\ln |\det (p_N(\tau )-z)|\right|\le  \mO(1), %
\end{equation}
hence
\begin{equation}\label{det.10}
\left| \ln |\det {\mathcal{T}}|-\ln |\det (p_N(\tau )-z)|\right|\le  \mO(1). %
\end{equation}
\subsection{The perturbed operator}
We next extend the estimates to the case of a perturbed operator
\begin{equation}\label{det.11}
P_N^\delta =P_N+\delta Q,
\end{equation}
where $Q:\ell^2(I_N)\to \ell^2(I_N)$ satisfies
\begin{equation}\label{det.12}
\delta \| Q\|\ll 1.
\end{equation}
\begin{remark}
	Later we will restrict in probability to the event that $\| Q\|_{\mathrm{HS}}\le \sqrt{C_1N}$ for a
	certain constant $C_1>0$ and we then require that $\delta \sqrt{N} \ll 1$
	which implies (\ref{det.12}), see also \eqref{mark1}.
\end{remark}
\begin{prop}\label{prop:det1}
	Let $K\Subset \C$ be an open relatively compact set and suppose 
	that \eqref{det.1}, \eqref{det.2} hold. 
	Recall \eqref{gr.4}, if $\delta \| Q\| \alpha^{-1} \ll 1$, then for all $z\in K\backslash\gamma_{\alpha}$
	\begin{equation}\label{det.13}
	\mathcal{P}_N^\delta =\begin{pmatrix}P_N^\delta -z &R_-\\ R_+
 	 &R_{+-}(z)\end{pmatrix}
	={\mathcal{P}}+\begin{pmatrix}\delta Q &0\\0 &0\end{pmatrix},
	\end{equation}
	is bijective with bounded inverse 
	\begin{equation}\label{det.13.1}
	\mathcal{E}_N^\delta =\begin{pmatrix}E^\delta &E^\delta_+\\ E^\delta_-
 	 &E^\delta_{+-}\end{pmatrix}.
	\end{equation}
	Recall \eqref{2gr.10.0}, if $\delta \| Q\| \alpha^{-2} \ll 1 $, then 
	for all $z\in K\backslash\gamma_{\alpha}$
	\begin{equation}\label{det.13.2}
	\mathcal{T}_N^\delta =\begin{pmatrix}P_N^\delta -z &T_-\\ T_+
 	 &T_{+-}\end{pmatrix}
	={\mathcal{T}_N}+\begin{pmatrix}\delta Q &0\\0 &0\end{pmatrix}.
	\end{equation}
	is bijective with bounded inverse 
	\begin{equation}\label{det.13.3}
	\mathcal{G}_N^\delta =\begin{pmatrix}G^\delta &G^\delta_+\\ G^\delta_-
 	 &G^\delta_{+-}\end{pmatrix},
	\end{equation}
	with 
	\begin{equation}\label{det.13.4}
		G_{-+}^\delta (z)=G_{-+}-G_-\delta Q(1+G\delta Q)^{-1}G_+.
	\end{equation}
	Moreover, $\| \mathcal{E}_N^\delta\| \leq 4/\alpha$, $\| \mathcal{G}_N^\delta\| \leq \mO(\alpha^{-2})$, 
	uniformly in $z\in K\backslash \gamma_{\alpha}$ and $N>1$.
\end{prop}
\begin{proof}
	We sometimes drop the subscript $N$. By \eqref{gr.10}, 
	\begin{equation*}
	\mathcal{P}^\delta \mathcal{E} = 1 +
	\begin{pmatrix}\delta Q E&\delta Q E_+\\0 &0\end{pmatrix}.
	\end{equation*}
	By \eqref{gr.12}, it follows that $\|E\| \leq 2/\alpha$, so if $\delta \| Q\| \alpha^{-1} \ll 1$, 
	then by Neumann series argument, the above is invertible and 
	\begin{equation}\label{det.13.5}
	\mathcal{E} \left(1 +
	\begin{pmatrix}\delta Q E&\delta Q E_+\\0 &0\end{pmatrix} \right)^{-1}
	\end{equation}
	is a right inverse of $\mathcal{P}^\delta$, of norm $\leq 2 \| \mathcal{E} \| \leq 4/\alpha$. 
	Since $\mathcal{P}^\delta$ is Fredholm of index $0$, this is also a left inverse. 
	The proof for $\mathcal{T}_N^\delta$ is similar, using that $\| G\| =\mO(\alpha^{-2})$ 
	by \eqref{2gr.10} , since $\tau>0$ is fixed. Finally, the expression \eqref{det.13.4} follows 
	easily from expanding \eqref{det.13.5}.
\end{proof}
We drop the subscript $N$ until further notice. By \eqref{det.13.2}, we have
\begin{equation}\label{det.14}
\|{\mathcal{T}}-{\mathcal{T}}^\delta \|_{\mathrm{tr}}\le \delta \| Q\|_{\mathrm{tr}}.
\end{equation}
Write,
$$
{\mathcal{T}}^\delta ={\mathcal{T}}(1-{\mathcal{T}}^{-1}({\mathcal{T}}-{\mathcal{T}}^\delta )),
$$
where
\begin{equation}\label{det.15}
\| {\mathcal{T}}^{-1}({\mathcal{T}}-{\mathcal{T}}^\delta )\|_{\mathrm{tr}}\le {\mathcal{O}
  }(\delta )\| Q\|_{\mathrm{tr}}.
\end{equation}
Here, we used that $\|{\mathcal{T}}^{-1}\| = \| \mathcal{G} \| = \mO(1)$, by 
\eqref{2gr.10} and the fact that $\tau >0$ is fixed. We recall that the estimates here depend on $\alpha$, 
yet are uniform in $z\in K\backslash \gamma_{\alpha}$ and $N>1$. It follows that
$$
| \det (1-{\mathcal{T}}^{-1}({\mathcal{T}}-{\mathcal{T}}^\delta ))|\le \exp \| {\mathcal{T}
  }^{-1}({\mathcal{T}}-{\mathcal{T}}^\delta )\|_{\mathrm{tr}}\le \exp ({\mathcal{O}
  }(\delta )\| Q\|_{\mathrm{tr}}),
$$
and
\begin{equation}\label{det.16}
\begin{split}
|\det {\mathcal{T}}_\delta | 
&=|\det {\mathcal{T}}| |\det (1-{\mathcal{T} }^{-1}({\mathcal{T}}-
	{\mathcal{T}}^\delta ))|\\ 
&\le \exp ({\mathcal{O}}(\delta )\|
	Q\|_{\mathrm{tr}})|\det {\mathcal{T}}|.
\end{split}
\end{equation}
\par 
Similarly from the identity
$$
{\mathcal{T}} ={\mathcal{T}}^\delta (1-{\mathcal{T}}_\delta ^{-1}({\mathcal{T}}^\delta-{\mathcal{T}} )),
$$
(putting $\delta $ as a subscript whenever convenient), we get
\begin{equation}\label{det.17}
|\det {\mathcal{T}} |\le \exp ({\mathcal{O}}(\delta )\|
Q\|_{\mathrm{tr}})|\det {\mathcal{T}}^\delta |,
\end{equation}
thus
\begin{equation}\label{det.18}
\left| \ln |\det {\mathcal{T}}_\delta |-\ln |{\mathcal{T}}| \right| \le {\mathcal{O}
  }(\delta )\| Q\|_{\mathrm{tr}}.
\end{equation}
Assume that (uniformly in $N>1$ and independently of $\alpha$)
\begin{equation}\label{det.19}
	\delta \| Q\|_{\mathrm{tr}}\le {\mathcal{O}}(1)
\end{equation}
and recall (\ref{det.10}). Then 
\begin{equation}\label{det.20}
\left| \ln |\det {\mathcal{T}}_\delta |-\ln |\det (p_N(\tau )-z)|\right|\le {\mathcal{O}}(1).
\end{equation}
Notice that the error term depends on $\alpha$. 
Using also the general identity (cf.\ (\ref{det.5})),
\begin{equation}\label{det.21}
  \begin{split}
  \det (P_N^\delta -z) = \det {\mathcal{T}}^\delta (z)\, \det G^\delta _{-+}(z),
\end{split}
\end{equation}
we get
\begin{equation}\label{det.22}
  \ln |\det (P_N^\delta -z)|=
  \ln |\det (p_N(\tau )-z)|+\ln |\det G_{-+}^\delta |+{\mathcal{O}}(1),
\end{equation}
uniformly for $z\in K\setminus \gamma _\alpha $, $N\gg 1$.

\section{Lower bounds with probability close to 1}\label{lb}
\setcounter{equation}{0}

\par We now adapt the discussion in \cite[Section 5]{SjVo19} to ${\mathcal{T}
  }^\delta $. Let
\begin{equation}\label{lb.1}
P_N^\delta =P_N+\delta Q_\omega ,\ \ Q_\omega =(q_{j,k}(\omega
))_{1\le j,k\le N},
\end{equation}
where $0\le \delta \ll 1$ and $q_{j,k}(\omega )\sim {\mathcal{N}}(0,1)$ are
independent normalized complex Gaussian random variables. By the Markov 
inequality, we have
\begin{equation}\label{lb.3}
{\bf P}[\| Q_\omega \|_{\mathrm{HS}}\le \sqrt{C_1N} ]\ge 1-e^{-N},
\end{equation}
for some universal constant $C_1>0$. In the following we restrict the
attention to the case when
\begin{equation}\label{lb.2}
\| Q_\omega \|_{\mathrm{HS}}\le \sqrt{C_1N}  ,
\end{equation}
and (as before) $z\in K\setminus \gamma _\alpha $, $N\gg 1$. 
We assume that
\begin{equation}\label{lb.5}
\delta \ll N^{-1}.
\end{equation}
Then
$$
\delta \| Q\|_{\mathrm{tr}}\le \delta N^{1/2}\| Q\|_{\mathrm{HS}}\le
\delta C_1N \ll 1,
$$
and the estimates of the previous sections apply.

\par Let ${\mathcal{Q}}_{C_1N}$ be the set of matrices satisfying
(\ref{lb.2}). As in \cite[Section 5.3]{SjVo19} we study the map \eqref{det.13.4}, i.e. 
\begin{equation}\label{lb.6}
\begin{split}
{\mathcal{Q}}_{C_1N}\ni Q\mapsto G_{-+}^\delta (z)
&=G_{-+}-G_-\delta
Q(1+G\delta Q)^{-1}G_+\\
&=G_{-+}-\delta G_-(Q+T(z,Q))G_+,
\end{split}
\end{equation}
where
\begin{equation}\label{lb.7}
T(z,Q)=\sum_1^\infty (-\delta )^nQ(GQ)^n,
\end{equation}
and notice first that by \eqref{2gr.10}
\begin{equation}\label{lb.8}
\| T\|_{\mathrm{HS}}\le {\mathcal{O}}(\delta\alpha^{-2}N ).
\end{equation}
We strengthen the assumption (\ref{lb.5}) to
\begin{equation}\label{lb.9}
  \delta \ll N^{-1}\alpha^2.
\end{equation}

\par At the end of Section \ref{2gr} we have established the uniform
injectivity and surjectivity respectively for $G_+$ and $G_-$. This
means that the singular values $s_j^\pm$ of $G_\pm$ for $1\le j\le
k(z)=\mathrm{rank\,}(G_-)=\mathrm{rank\,}(G_+) $ satisfy
\begin{equation}\label{lb.10}
\frac{1}{C}\le s_j^\pm \le \frac{2}{\alpha}
\end{equation}
This corresponds to \cite[(5.27)]{SjVo19} and
the subsequent discussion there carries over to the present
situation with the obvious modifications. Similarly to \cite[(5.42)]{SjVo19} we 
strengthen the assumption on $\delta $ to
\begin{equation}\label{lb.11}
\delta \ll N^{-3/2}\alpha^2
\end{equation}
\begin{remark}
	If following exactly the proof in \cite[Section 5.3]{SjVo19} we would need to suppose 
	that $\delta \ll N^{-3}\alpha^2$. However, it is sufficient for our purposes to work with 
	$\| Q_\omega \|_{\mathrm{HS}}\le \sqrt{C_1N} $ which holds with probability \eqref{lb.3} 
	instead of $\| Q_\omega \|_{\mathrm{HS}}\le C_1N $ as we did in \cite[Section 5.3]{SjVo19}. 
	This change allows us to make the weaker assumption \eqref{lb.11} while still being able to 
	follow the proof in \cite[Section 5.3]{SjVo19}.
\end{remark}
Notice, that assumption \eqref{lb.11} is stronger than the assumptions on $\delta $ 
in Proposition \ref{prop:det1}. 
The same reasoning as in \cite[Section 5.3]{SjVo19}  leads to the following adaptation of
Proposition 5.3 in \cite{SjVo19}:
\begin{prop}\label{lb1}
Let $K\subset {\bf C}$ be compact, $0<\alpha \ll 1$ and choose $M$ so
that $\epsilon (M)\le \alpha /2$. Let $\delta $ satisfy
(\ref{lb.11}). Then the second Grushin problem with matrix ${\mathcal{T}
  }^\delta $ is
well posed with a bounded inverse ${\mathcal{G}}^\delta $
introduced in Proposition \ref{prop:det1}. The following holds uniformly for
$z\in K\setminus \gamma _\alpha $, $N\gg 1$:

\par There exist positive constants $C_0$, $C_2$ such that
$$
{\bf P}\left(\ln |\det G_{-+}^\delta (z)|^2\ge -t\hbox{ and }\|
Q\|_{\mathrm{HS}}\le C_1\sqrt{N}\right)
\ge 1-e^{-N}-C_2\delta ^{-M}e^{-t/2},
$$
when
$$
t\ge C_0-2M \ln \delta ,\ \ 0<\delta \ll N^{-3/2}\alpha^2.
$$
\end{prop}

\section{Counting eigenvalues in smooth domains}\label{ce}
\setcounter{equation}{0}

\par 
We work under the assumptions of Proposition \ref{lb1} and 
from now on we assume that $\delta$ satisfies \eqref{m0}, i.e. 
\begin{equation}\label{ce.0}
	\e^{- N^{\delta_0} } \leq \delta \ll N^{-\delta_1},
\end{equation}
for some fixed $\delta_0 \in ]0,1[$ and $\delta_1 >3/2$. 
Notice that \eqref{lb.11} holds for $N>1$ sufficiently large (depending on $\alpha$). 
Then with 
probability $\ge 1-e^{-N}$ we have $G_{-+}^\delta (z)={\mathcal{O}}(1)$
for every $z\in K\setminus \gamma _\alpha $, hence by (\ref{det.22})
\begin{equation}\label{ce.1}
\ln |\det (P_N^\delta -z)|\le \ln |\det (p_N(\tau )-z)|+{\mathcal{O}}(1).
\end{equation}
On the other hand, by (\ref{det.22}) and Proposition \ref{lb1}, we
have for every $z\in K\setminus \gamma _\alpha $ that
\begin{equation}\label{ce.2}
\ln |\det (P_N^\delta -z)|\ge \ln |\det (p_N(\tau
)-z)|-\frac{t}{2}-{\mathcal{O}}(1)
\end{equation}
with probability
\begin{equation}\label{ce.3}
\ge 1-e^{-N}-C_2\delta ^{-M}e^{-t/2},
\end{equation}
when
\begin{equation}\label{ce.4}
  t\ge C_0-2M\ln \delta .
\end{equation}

\par Define $\phi (z)=\phi _N(z)$ by requiring that
\begin{equation}\label{ce.5}
N\phi (z)=\ln |\det (p_N(\tau )-z)| \hbox{ on } K\setminus \gamma _\alpha ,
\end{equation}
and
\begin{equation}\label{ce.6}
\phi (z)\hbox{ is continuous in }K\hbox{ and harmonic in
}\stackrel{\circ }{\gamma} _\alpha 
\end{equation}
Here it is understood that we may enlarge $\gamma _\alpha $ away from
a neighborhood of
the most interesting region $\partial \Omega \cap \gamma $ and achieve
that $\partial \gamma _\alpha $ be smooth everywhere.
\par 
Choose
\begin{equation}\label{ce.7}
t=N^{\epsilon _0},
\end{equation}
for some fixed $\epsilon _0\in ]0,1[$ with $\delta_0< \varepsilon_0$, 
see \eqref{ce.0}, \eqref{m0}. Then
$$
C_2\delta ^{-M}e^{-t/2}=\exp \left( \ln C_2 -M\ln \delta
  -N^{\epsilon _0}/2\right),
$$
and we require from $\delta $ that 
$$
\ln C_2 -M\ln \delta -N^{\epsilon _0}/2\le -N^{\epsilon _0}/4,
$$
i.e.
$$
\ln \delta \ge \frac{\ln C_2 }{M} -\frac{N^{\epsilon _0}}{4M}.
$$
This is fulfilled if $N\gg 1$ and
$$
\ln \delta \ge -\frac{N^{\epsilon _0}}{5M},
$$
i.e.
\begin{equation}\label{ce.8}
\delta \ge \exp \left(-\frac{1}{5M} N^{\epsilon _0} \right)
\end{equation}
and (\ref{ce.7}), (\ref{ce.8}) imply (\ref{ce.4}) when $N\gg 1$. Notice that \eqref{ce.0} implies \eqref{ce.8} 
for $N\gg 1$. 

Combining (\ref{ce.2}), (\ref{ce.5}), (\ref{ce.7}) and (\ref{ce.8}),
we get for each $z\in K\setminus \gamma _\alpha $ that
\begin{equation}\label{ce.9}
\ln |\det (P_N^\delta -z)|\ge N(\phi (z)-\epsilon _1),
\end{equation}
with probability
\begin{equation}\label{ce.10}
\ge 1-e^{-N}-e^{-N^{\epsilon _0}/4} 
\end{equation}
where
\begin{equation}\label{ce.11}
\epsilon _1=N^{\epsilon _0-1}.
\end{equation}
Here and in the following, we assume that $N\ge N(\alpha ,K)$
sufficiently large.

\par On the other hand, with probability $\ge 1-e^{-N}$, we have
\begin{equation}\label{ce.12}
\ln |\det (P_N^\delta -z)|\le N(\phi (z)+\epsilon _1)
\end{equation}
for all $z\in K\setminus \gamma _\alpha $. We assume that $K$ is large
enough to contain a neighborhood of $\gamma _\alpha $. Then, since the
left hand side in (\ref{ce.12}) is subharmonic and the right hand side
is harmonic in $\gamma _\alpha $, we see that (\ref{ce.12}) remains
valid also in $\gamma _\alpha $ and hence in all of $K$.

\par Let $\Omega \Subset {\bf C}$ be as in Theorem \ref{main}, so that
$\partial \Omega $ intersects $\gamma $ at finitely many points
$\widetilde{z}_1,...,\widetilde{z}_{k_0}$ which are not critical
values of $p$ and where the intersection is transversal. Choose
$z_1,...,z_L\in \partial \Omega \setminus \gamma _\alpha $ such that
with $r_0=C_0\alpha $, $C_0\gg 1$, we have
\begin{equation}\label{ce.13}
\frac{r_0}{4}\le |z_{j+1}-z_j|\le \frac{r_0}{2}
\end{equation}
where the $z_j$ are distributed along the boundary in the positively
oriented sense and with the cyclic convention that $z_{L+1}=z_{1}$. 
Notice that $L={\mathcal{O}}(1/\alpha )$. Then
$$\partial \Omega \subset \bigcup_{1}^L D(z_j,r_0/2)$$ 
and we can arrange
so that $z_j\not\in \gamma _\alpha $ and even so that
\begin{equation}\label{ce.14}
\mathrm{dist\,}(z_j,\gamma _\alpha )\ge \alpha .
\end{equation}

\par Choose $K$ above so that $\overline{\Omega }\Subset K$. Combining 
\eqref{ce.12} and \eqref{ce.9} we have that $\det (P_N^\delta -z)$ satisfies the upper bound \eqref{ce.12} 
for all $z\in K$ and the lower bound \eqref{ce.9} for $z = z_1,\dots, z_L$ with probability 
\begin{equation}\label{ce.15.5}
\ge 1-\mO(\alpha^{-1})(e^{-N}+e^{-N^{\epsilon _0}/4}).
\end{equation}
Since
$\phi $ is continuous and subharmonic, we can apply the zero counting
theorem of \cite{Sj09b} (see also \cite[Chapter 12]{Sj19}) to the
holomorphic function $\det (P_N^\delta -z)$ and get
\begin{multline}\label{ce.15}
  \left|
\# (\sigma (P_N^\delta )\cap \Omega )-\frac{N}{2\pi }\int_{\Omega
}\Delta \phi L(dz)   \right|\le {\mathcal{O}}(N)\times \\
\left( L\epsilon _1 +\int_{\partial \Omega +D(0,r_0 )}\Delta \phi
  L(dz)+\sum_1^L \int_{D(z_j,r_0)} \Delta \phi (z) \left| \ln \frac{|z-z_j|}{r_0} \right|L(dz) \right)
\end{multline}
with probability \eqref{ce.15.5}.
\\
\par
Recall that $L={\mathcal{O}}(1/\alpha )$ (hence ${\mathcal{O}}(1)$ for every fixed $\alpha $). $\Delta \phi
$ is supported in $\gamma _\alpha $ and the number of discs
$D(z_j,r_0)$ that intersect $\gamma _\alpha $ is $\le {\mathcal{O}}(1)$
uniformly with respect to $\alpha $. Also $\ln (|z-z_j|/r_0)={\mathcal{O}
  }(1)$ on the intersection of each such disc with $\gamma _\alpha
$. Since $\epsilon _1=N^{\epsilon _0-1}$, we get from
(\ref{ce.15}):

\begin{multline}\label{ce.16}
\left| \# (\sigma (P_N^\delta )\cap \Omega )-\frac{N}{2\pi }\int_{\Omega
  }\Delta \phi L(dz)  \right|\\
\le {\mathcal{O}}(N)\left( {\mathcal{O}}_\alpha (N^{\epsilon _0-1})+\int_{(\gamma \cap \partial
    \Omega )+D(0,2r_0)} \Delta \phi (z) L(dz) \right).
\end{multline}

\par We next study the measure $\Delta \phi $. Away from $\gamma \cap
\partial \Omega $ we may enlarge $\gamma _\alpha $ to become a nice
domain with smooth boundary everywhere. Notice that
\begin{equation}\label{ce.17}
  \ln |\det (p_N(\tau )-z)|=\sum_1^{N+M}\ln |z-\lambda _j|,
\end{equation}
where
$$
\lambda _j=p\left(\exp \frac{2\pi ij}{N+M} \right),\ 1\le j\le N+M,
$$
and this expression is equal to $N\phi (z)$ in $K\setminus \gamma
_\alpha $.

\par Define
\begin{equation}\label{ce.18}
\psi (z)=\phi (z)-\frac{1}{N}\sum_1^{N+M}\ln |z-\lambda _j|,
\end{equation}
so that $\psi $ is continuous away from the $\lambda _j\in \gamma $,
\begin{equation}\label{ce.19}
\psi (z)=0\hbox{ in }{\bf C}\setminus \gamma _\alpha ,
\end{equation}
\begin{equation}\label{ce.20}
{{\psi }_\vert}_{\partial \gamma _\alpha }=0,
\end{equation}
\begin{equation}\label{ce.21}
\Delta \psi =-\frac{2\pi }{N}\sum_1^{N+M}\delta _{\lambda _j}\hbox{ in
}\stackrel{\circ }{\gamma }_\alpha .
\end{equation}
It follows that in $\gamma _\alpha $:
\begin{equation}\label{ce.22}
\psi (z)=-\frac{2\pi }{N}\sum_1^{N+M} G_{\gamma _\alpha }(z,\lambda _j),
\end{equation}
where $G_{\gamma _\alpha }$ is the Green kernel for $\gamma _\alpha $.

\par $\phi $ is harmonic away from $\partial \gamma _\alpha $, so for
$\phi $ as a distribution on ${\bf C}$, we have $\mathrm{supp\,}\Delta
\phi \subset \partial \gamma _\alpha $. Now $\psi -\phi $ is harmonic
near $\partial \gamma _\alpha $, so $\Delta \psi =\Delta \phi $ near
$\partial \gamma _\alpha $. In the interior of $\gamma _\alpha $ we
have (\ref{ce.21}) and in order to compute $\Delta \psi$ globally, we
let $v\in C_0^\infty ({\bf C})$ and Green's formula to get
\begin{multline*}
\langle \Delta \psi ,v\rangle =\langle \psi ,\Delta v\rangle
=\int_{\gamma _\alpha }\psi \Delta v L(dz)\\
=\int_{\gamma _\alpha }\Delta \psi v L(dz)+\int_{\partial \gamma
  _\alpha }\psi \partial _{\nu }v |dz|-\int_{\partial \gamma _\alpha
}\partial _\nu \psi v |dz|.
\end{multline*}
Here $\nu $ is the exterior unit normal and in the last term it is
understood that we apply $\partial _\nu $ to the restriction of $\psi
$ to $\stackrel{\circ }{\gamma }_\alpha $ then take the boundary
limit. (\ref{ce.20}), (\ref{ce.21}), (\ref{ce.22}) imply that in the
sense of distributions on ${\bf C}$,
\begin{equation}\label{ce.23}
\Delta \psi =- \frac{2\pi}{N}\sum_1^{N+M}\delta _{\lambda _j}
+\frac{2\pi }{N}\partial _\nu \left(\sum_1^{N+M}G_{\gamma _\alpha
  }(\cdot ,\lambda _j) \right)L_{\partial \gamma _\alpha }(dz)
\end{equation}
where $L_{\partial \gamma _\alpha }$ denotes the (Lebesgue) arc length
measure supported on $\partial \gamma _\alpha $. 
\par 
By the preceding discussion we conclude that
\begin{equation}\label{ce.24}
\Delta \phi =\frac{2\pi }{N} \left(\sum_1^{N+M}\partial _\nu G_{\gamma _\alpha
  }(\cdot ,\lambda _j) L_{\partial \gamma _\alpha }(dz) \right).
\end{equation}
Each term in the sum is a non-negative measure of mass 1:
\begin{equation}\label{ce.25}
\int \partial _\nu G(z ,\lambda _j)L_{\partial \gamma _\alpha }(dz)=1.
\end{equation}
%
%From the estimates in the proof of the counting theorem in \cite{Sj10}
%(\cite{Sj19}) we know that
\begin{lemma}
We have 
\begin{equation}\label{ce.26}
\left| \partial _\nu G_{\gamma _\alpha }(z ,\lambda _j) \right|\le
\frac{1}{\alpha }e^{-\frac{|z-\lambda _j|}{C\alpha }},
\end{equation}
for $z\in \partial \gamma _\alpha \cap \mathrm{neigh\,}(\gamma  \cap \partial\Omega
)$, $\lambda _j\in \gamma _\alpha $, $|z-\lambda _j|\ge \alpha /{\mathcal{O}
  }(1)$. (\ref{ce.26}) also holds when $z\in \partial \gamma _\alpha
,$ $\lambda _j\in \gamma _\alpha $, $|z-\lambda _j|\ge \alpha /{\mathcal{O}
  }(1)$ and $(z,\lambda _j)\in (\Omega \times ({\bf C}\setminus
\Omega ))\cup (({\bf C}\setminus \Omega )\times \Omega ) $. 
\end{lemma}
\begin{proof} 
By scaling of the harmonic function $G_{\gamma_\alpha}(\cdot,\lambda_j)$ by a 
factor $1/\alpha$, it suffices to show that 
\begin{equation}\label{ce.26.5}
\left|  G_{\gamma _\alpha }(z ,\lambda _j) \right|\le
e^{-\frac{|z-\lambda _j|}{C\alpha }},
\end{equation}
for $(z,\lambda_j)$ as after \eqref{ce.26} with the difference that $z$ now varies in $\gamma_{\alpha}$ 
instead of $\partial\gamma_{\alpha}$.
\par
We decompose $\gamma
_\alpha $ as $\Gamma _1\cup \Gamma _2\cup \gamma _{1,\alpha }\cup
...\cup \gamma _{T,\alpha }$, where $\Gamma _1$ and $\Gamma _2$ are
the enlarged parts of $\gamma _\alpha $ with $\Gamma _1\subset \Omega
$, $\Gamma _2\subset {\bf C}\setminus \Omega $ and $\gamma _{1,\alpha
},...,\gamma _{T,\alpha }$ are the regular parts of width $2\alpha $,
corresponding to the segments of $\gamma $, that intersect $\partial
\Omega $ transversally. 

\par For simplicity, we assume that $\Gamma _1$ and $\Gamma _2$ are
connected and that each segment $\gamma _{k,\alpha }$ links
$\Gamma _1$ to $\Gamma _2$ and crosses $\partial \Omega $ once. We may
think of $\gamma _{\alpha }$ as a graph with the vertices $\Gamma _1$,
$\Gamma _2$ and with $\gamma _{k,\alpha }$ as the edges.

\par Let first $\lambda _j$ belong to $\Gamma _1$. We apply the first
estimate in Proposition 2.2 in \cite{Sj09b} or equivalently Proposition
12.2.2 in \cite{Sj19} and see that
$-G_{\gamma _\alpha }(z,\lambda _j)\le {\mO}(1)$ for
$z\in \gamma _\alpha $, $|z-\lambda _j|\ge 1/{\mO}(1)$.
Possibly, after cutting away a piece of $\gamma _{k,\alpha }$ and
adding it to $\Gamma _1$, we may assume that $-G_{\gamma _\alpha
}(z,\lambda _j)\le {\mO}(1)$ in $\gamma _{k,\alpha }$.
Consider
one of the $\gamma _{k,\alpha }$ as a
finite band with the two ends given by the closure of the set of
$z\in \partial \gamma _{k,\alpha }$ with
$\mathrm{dist\,}(z,\partial \gamma _\alpha )<\alpha $. Let
$G_{\gamma _{k,\alpha }}$ denote the Green kernel of
$\gamma _{k,\alpha }$.  Then the second estimate in the quoted
proposition applies and we find
$$
-G_{\gamma _{k,\alpha }}(x,y)\le {\mO}(1)e^{-|x-y|/{\mO}(\alpha
  )},\hbox{ when } x,y\in \gamma _{k,\alpha },\ |x-y|\ge \alpha /{\mO}(1).
$$
Let
$$
u=\chi {{G_{\gamma _\alpha }(\cdot ,\lambda _j)}_\vert}_{\gamma
  _{k,\alpha }},
$$
where $\chi \in C^\infty (\gamma _{k,\alpha };[0,1])$ vanishes near
the ends of $\gamma _{k,\alpha }$, is equal to 1 away from an
$\alpha $-neighborhood of these end points and with the property that
$\nabla \chi ={\mO}(1/\alpha )$,
$\nabla ^2\chi ={\mO}(1/\alpha ^2)$. Then
${{u}_\vert}_{\partial \gamma _{k,\alpha }}=0$ and
$\Delta u={\mO}(\alpha ^{-2})$ is supported in an
$\alpha $-neighborhood of the union of the two ends and hence of
uniformly bounded $L^1$-norm. Now we apply the second estimate in the
quoted proposition to
$ u=\int G_{\gamma _{k,\alpha }}(\cdot ,y)\Delta u(y)L(dy) $ and we see that
\begin{equation}\label{1}
G_{\gamma _\alpha }(\cdot ,\lambda _j)={\mO}(e^{-1/{\mO}(\alpha )}).
\end{equation}
in $\{ x\in \gamma _{k,\alpha };\,\mathrm{dist\,}(x,\partial \Omega
\cap \gamma _{k,\alpha })\le 1/{\mO}(1) \}$.
Varying $k$, we get (\ref{1}) in $\{ x\in \gamma _\alpha ;\,
\mathrm{dist\,}(x,\partial \Omega \cap \gamma )\le 1/{\mO}(1)
\}$. Applying the maximum principle to the harmonic
function ${{G_{\gamma _\alpha }(\cdot ,\lambda _j)}_\vert}_{({\bf
    C}\setminus \Omega )\cap \gamma _\alpha }$, we see that (\ref{1})
holds uniformly in $({\bf C}\setminus \Omega )\cap \gamma _\alpha $.

\par Similarly, we have (\ref{1}) uniformly in
$$\{x\in \gamma _\alpha ;\, \mathrm{dist\,}(x,\partial \Omega \cap
\gamma )\le 1/{\mO}(1) \}\cup (\Omega \cap \gamma _\alpha ),$$ when
$\lambda _j\in \Gamma _2$ and we have shown \eqref{ce.26.5}, \eqref{ce.26} when
$\lambda _j\in \Gamma _1\cup \Gamma _2$. Similarly, we have
\eqref{ce.26} when $\lambda _j\in \gamma _{k,\alpha }$ is close to one
of the ends.

It remains to treat the case when $\lambda _j\in \gamma _{k,\alpha }$
is at distance $\ge 1/{\mO}(1)$ from the ends of $\gamma _{k,\alpha
}$. Defining $u=\chi {{G_{\gamma _\alpha }(\cdot ,\lambda
    _j)}_\vert}_{\gamma _{k,\alpha }}$ as before we now have
$$
\Delta u=[\Delta ,\chi ]G_{\gamma _\alpha (\cdot ,\lambda _j)}+\delta
_{\lambda _j},
$$
where the first term in the right hand side has its support in an
$\alpha $-neighborhood of the union of the ends and is ${\mO}(1)$
in $L^1$. By the second part of the quoted proposition we have
\begin{equation}\label{2}
u(x)={\mO}(1)\exp \left( -\frac{1}{{\mO}(\alpha )}\min
  \left(\mathrm{dist\,} (x,\mathrm{ends\,}(\gamma _{k,\alpha
    })),|x-\lambda _j| \right) \right),
\end{equation}
away from an $\alpha $-neighborhood of $\mathrm{ends\,}(\gamma
_{k,\alpha })\cup \{\lambda _j\}$. Here $\mathrm{ends\,}(\gamma
_{k,\alpha })$ denotes the union of the two ends of $\gamma _{k,\alpha
}$. Since $u$ is harmonic away from $\lambda _j$ and from $\alpha
$-neighborhoods of the ends, we get from (\ref{2}) that
\begin{equation}\label{3}
\nabla u(x)={\mO}\left(\frac{1}{\alpha } \right)\exp \left( -\frac{1}{{\mO}(\alpha )}\min
  \left(\mathrm{dist\,} (x,\mathrm{ends\,}(\gamma _{k,\alpha
    })),|x-\lambda _j| \right) \right),
\end{equation}
which gives \eqref{ce.26} near $\partial \Omega \cap \gamma $.
By using the maximum principle as before, we can extend the validity
of \eqref{ce.26} to all of $\partial \gamma _\alpha \setminus
D(\lambda _j,\alpha /{\mO}(1))$.
\end{proof}

\par
Continuing, notice that by \eqref{gr.3}, \eqref{ce.17} 
\begin{equation}\label{n1}
	\#\{ \sigma(P_{S_{\widetilde{N}}}) \cap \eta \} = 
	\#\{\widehat{S}_{\widetilde{N}} \cap p_N^{-1}(\eta) \}, \quad \widetilde{N} = N + M, 
\end{equation}
for $\eta \subset \gamma$. 
Since two consecutive points of $\widehat{S}_{\widetilde{N}}$ differ by an angle 
of $2\pi/\widetilde{N}$ and by the assumptions (1)-(4) prior to Theorem \ref{main}, we get that
$$
\# \{ \lambda _j;\, \mathrm{dist\,}(\lambda _j,\partial \Omega \cap
\gamma )<4r_0 \} =\mO(\alpha N)
$$
and also 
$$\# \{ \lambda _j;\, \mathrm{dist\,}(\lambda _j,\partial \Omega \cap
\gamma )\in [2^kr_0,2^{k+1}r_0[ \} =\mO(\alpha 2^k N),\ k=2,3,...
$$
From \eqref{ce.26} we get
\begin{equation}\label{ce.27}
\begin{split}
\frac{N}{2\pi }\int_{(\partial \Omega \cap \gamma )+D(0,2r_0)}\Delta \phi
L(dz)
&=\sum_j \int_{((\partial \Omega \cap \gamma )+D(0,2r_0))\cap
\partial \gamma _\alpha }\partial _\nu G_{\gamma _\alpha }(z,\lambda
_j) L(dz)\\
&=\mO(\alpha N)+\sum_{k=2}^\infty  \sum_{\lambda _j;\atop
  \mathrm{dist\,}(\lambda _j,\partial \Omega \cap \gamma )\in
  [2^kr_0,2^{k+1}r_0[}e^{-2^k/\mO(1)}\\
&=\mO(1)\left
  (\alpha N+\sum_{k=2}^\infty e^{-2^k/\mO(1)} \alpha 2^kN\right)\\ 
 &=\mO(\alpha N)+\mO(1)N\alpha \int_0^\infty e^{-t/{\mO}(1)}dt \\
&=\mO(\alpha N).
\end{split}
\end{equation}
Combining \eqref{ce.25} and \eqref{ce.26}, we get when $\mathrm{dist\,}(\lambda _j,\partial
\Omega \cap \gamma )\ge 2r_0$:
$$
\int_{\partial \gamma _\alpha \cap\Omega }\partial _\nu G_{\gamma
  _\alpha }(z,\lambda _j)L_{\partial \gamma _\alpha }(dz)
=\begin{cases}
1+\mO(1)e^{-\mathrm{dist\,}(\lambda _j,\partial \Omega \cap
  \gamma )/\mO(\alpha )},&\hbox{when }\lambda _j\in \Omega ,\\
\mO(1)e^{-\mathrm{dist\,}(\lambda _j,\partial \Omega \cap
  \gamma )/\mO(\alpha )},&\hbox{when }\lambda _j\not\in \Omega .
\end{cases}
$$
We now get
\begin{equation}\label{ce.28}
\begin{split}
    \frac{N}{2\pi }\int_{\Omega }\Delta \phi L(dz)=&\sum_{j;\,
      \mathrm{dist\,}(\lambda _j,\gamma \cap \partial \Omega )\le
      4r_0} \int_{\partial \gamma _\alpha \cap\Omega }\partial
    _\nu
    G_{\gamma _\alpha }(z,\lambda _j)L_{\partial \gamma _\alpha }(dz)\\
    &+\sum_{k=2}^\infty \sum_{\lambda _j\in \Omega , \atop
      \mathrm{dist\,}(\lambda _j,\gamma \cap \partial \Omega )\in
      [2^kr_0,2^{k+1}r_0[}\int_{\partial \gamma _\alpha \cap \Omega
    }\partial _\nu G_{\gamma _\alpha }(z,\lambda _j)L_{\partial \gamma
      _\alpha }(dz) \\
    &+\sum_{k=2}^\infty \sum_{\lambda _j\in {\bf C}\setminus \Omega ,
      \atop \mathrm{dist\,}(\lambda _j,\gamma \cap \partial \Omega
      )\in [2^kr_0,2^{k+1}r_0[}\int_{\partial \gamma _\alpha \cap
      \Omega }\partial _\nu G_{\gamma _\alpha }(z,\lambda
    _j)L_{\partial \gamma
      _\alpha }(dz)\\
    =& \mO(\alpha N)+ \sum_{k=2}^\infty \sum_{\lambda _j\in \Omega ,
      \atop \mathrm{dist\,}(\lambda _j,\gamma \cap \partial \Omega
      )\in
      [2^kr_0,2^{k+1}r_0[} (1+\mO(1)e^{-2^k/\mO(1)}) \\
    &+\sum_{k=2}^\infty \sum_{\lambda _j\in {\bf C}\setminus \Omega ,
      \atop \mathrm{dist\,}(\lambda _j,\gamma \cap \partial \Omega
      )\in
      [2^kr_0,2^{k+1}r_0[}\mO(1)e^{-2^k/\mO(1)}\\
    &=\#\{ \lambda _j\in \Omega \}+\mO(\alpha N).
  \end{split}
\end{equation}
Thus (\ref{ce.16}) gives 
\begin{equation}\label{ce.29}
\begin{split}
\# (\sigma (P_N^\delta )\cap \Omega )&=\# (\{\lambda _j \}\cap \Omega
)+{\mathcal{O}}(\alpha N)+{\mathcal{O}}_\alpha (N^{\epsilon _0})\\
&=\frac{N}{2\pi }\left( \int_{S^1\cap p^{-1}(\Omega )}L_{S^1}(d\theta) \right)
+{\mathcal{O}}(\alpha N )+{\mathcal{O}}_\alpha (N^{\epsilon _0})+o(N),
\end{split}
\end{equation}
with a probability as in (\ref{ce.15.5}) which is bounded from below by the probability 
\eqref{m2} for $N>1$ sufficiently large. Here and in the next formula we view $p_N$ and $p$ as maps from  
$S^1$ to $\mathbf{C}$. 
In the second equality we used that by \eqref{n1}
 \begin{equation}\label{n2}
 \begin{split}
	\# (\{\lambda _j \}\cap \Omega)
	&=\frac{\widetilde{N}}{2\pi} \int_{S^1\cap p_N^{-1}(\Omega)}L_{S^1}(d\theta) +\mO(1) \\ 
	&=\frac{N}{2\pi} \int_{S^1\cap p_N^{-1}(\Omega )}L_{S^1}(d\theta) +\mO(M) \\ 
	&=\frac{N}{2\pi} \int_{S^1\cap p^{-1}(\Omega )}L_{S^1}(d\theta) +o(N),
\end{split}
\end{equation}
where we used that $p_N \to p$ uniformly on $S^1$ and where the measure $L_{S^1}(d\theta)$ in 
the integral denotes the Lebesgue measure on $S^{1}$. 
\par
Theorem \ref{main} follows by taking $\alpha>0$ in \eqref{ce.29} arbitrarily small and $N>1$ 
sufficiently large.
\section{Convergence of the empirical measure}
In this section we present a proof of Theorem \ref{main2} following the strategy 
of \cite[Section 7.3]{SjVo19}. An alternative, and perhaps more direct way,  
to conclude the weak convergence of the empirical measure from a counting 
theorem as Theorem \ref{main2}, has been presented in \cite[Section 7.1]{SjVo19}. 
\par
Recall the definition of the empirical measure $\xi_N$ \eqref{emp1}. 
By \eqref{mark1}, \eqref{unp.5.1} combined with a Borel Cantelli argument, it follows that 
almost surely 
\begin{equation}\label{emc.1}
	\supp \xi_N \subset  \overline{D(0,\| p \|_{L^{\infty}(S^1)}+1)} \defeq K\subset 
	D(0,\| p \|_{L^{\infty}(S^1)}+2)\defeq K'
\end{equation}
for $N$ sufficiently large. For $p$ as in \eqref{unp.5}, put 
\begin{equation}\label{eq:em2.4}
	 \xi =p_*\left(\frac{1}{2\pi} L_{S^1}\right)
\end{equation}
which has compact support, 
\begin{equation}\label{eq:em2.5}
	\supp \xi = p(S^1) \subset K.
\end{equation}
Here, $\frac{1}{2\pi} L_{S^1}$ denotes the normalized Lebesgue measure on $S^1$. 
\\
\par
Using \cite[Theorem 7.1]{SjVo19}, it remains to show that for almost every $z\in K'$ we 
have that $U_{\xi_N}(z) \to U_{\xi}(z)$ almost surely, where 
\begin{equation*}
	U_{\xi_N}(z) = - \int \log | z- x | \xi_N(dx), \quad U_{\xi}(z) = - \int \log | z- x | \xi(dx).
\end{equation*}
\par
The cited Theorem is a modification 
of a classical result which allows to deduce the weak convergence of measures from the point-wise 
convergence of the associated Logarithmic potentials, see for instance \cite[Theorem 2.8.3]{Ta12} or 
\cite{BoCh13}.
\\
\par
For $z\notin \sigma( P_N^{\delta})$ 
\begin{equation}\label{emc.3}
	U_{\xi_N}(z) = -\frac{1}{N} \log| \det (P_N^{\delta} -z)|.
\end{equation}
For any $z\in \C$ the set $\Sigma_z = \{ Q \in \C^{N\times N}; \det(P_N +\delta Q -z) =0\}$ 
has Lebesgue measure $0$, since $\C^{N\times N} \ni Q \mapsto \det (P_N^{\delta} -z)$ is 
analytic and not constantly $0$. Thus $\mu_N(\Sigma_z)=0$, where $\mu_N$ is the Gaussian 
measure given in after \eqref{unp7.3}, and for every $z\in \C$ \eqref{emc.3} holds 
almost surely. 
\par
Let $\delta$ satisfy \eqref{m0} for some fixed $\delta_0\in ]0,1[$ and $\delta_1>3/2$. Pick 
a $\varepsilon_0 \in ]\delta_0,1[$. Let $z\in K'\backslash p(S^1)$. Recall \eqref{det.2}. 
For $\alpha>0$ sufficiently small, we have that $z\in K'\backslash \gamma_{\alpha}$. 
By taking $N>1$ sufficiently large, we have that $p(S^1) \subset \gamma_{\alpha/2}$.  

Then, by \eqref{ce.9} and \eqref{ce.12}, 
\begin{equation}\label{emc.4}
	\left| 
	\frac{1}{N} \log| \det (P_N^{\delta} -z)| - \phi(z)
	\right|
	\leq \mO(N^{\varepsilon_0 -1}).
\end{equation}
with probability $\geq 1 - \e^{-N}- \e^{-N^{\varepsilon_0/4}}$. Here, 
$\phi(z) = N^{-1}\ln |\det (p_N(\tau )-z)|$ by \eqref{ce.5}, since $z\notin \gamma_\alpha$.
\par
Using a Riemann sum argument and the fact that $p_N \to p$ uniformly on $S^1$, we have that 
\begin{equation}\label{emc.5}
	\left| 
	\phi(z) + U_{\xi}(z) 
	\right| \longrightarrow 0, \quad \hbox{as } N\to \infty.
\end{equation}
Thus, by \eqref{emc.4}, \eqref{emc.5}, we have for any $z\in K'\backslash p(S^1)$ that 
\begin{equation}
	\left| U_{\xi_N}(z) - U_{\xi}(z) 	\right| = o(1)
\end{equation}
with probability $\geq 1 - \e^{-N}- \e^{-N^{\varepsilon_0/4}}$. By the Borel-Cantelli theorem 
if follows that for every $z\in K'\backslash p(S^1)$ 
\begin{equation}
	U_{\xi_N}(z)   \longrightarrow  U_{\xi}(z), \quad \hbox{as } N\to \infty, \hbox{ almost surely},
\end{equation}
which by \cite[Theorem 7.1]{SjVo19} concludes the proof of Theorem \ref{main2}.


\begin{thebibliography}{30}

\bibitem[BoCa13]{BoCh13}
C.~Bordenave and D.~Chafa{\"\i}.
\newblock Lecture notes on the circular law.
\newblock In V.~H. Vu, editor, {\em Modern Aspects of Random Matrix Theory},
  volume~72, pages 1--34. Amer. Math. Soc., 2013.

\bibitem[BaPaZe18a]{BPZ18}
A.~Basak, E.~Paquette, and O.~Zeitouni.
\newblock Regularization of non-normal matrices by gaussian noise - the banded
  toeplitz and twisted toeplitz cases.
\newblock {\em Forum Math. Sigma 7}, e3, 2019. 

\bibitem[BaPaZe18b]{BPZ18b}
A.~Basak, E.~Paquette, and O.~Zeitouni.
\newblock Spectrum of random perturbations of toeplitz matrices with finite
  symbols.
\newblock {\em preprint https://arxiv.org/pdf/1812.06207.pdf}, 2018.

\bibitem[B{\"o}Si99]{BoSi99}
A.~B\"ottcher and B.~Silbermann.
\newblock {\em Introduction to large truncated Toeplitz matrices}.
\newblock Springer, 1999.

\bibitem[Da07]{Da07}
E.~B. Davies.
\newblock {\em {Non-Self-Adjoint Operators and Pseudospectra}}, volume~76 of
  {\em {Proc. Symp. Pure Math.}}
\newblock Amer. Math. Soc., 2007.

\bibitem[DaHa09]{DaHa09}
E.B. Davies and M.~Hager.
\newblock {\em Perturbations of Jordan matrices}.
\newblock {\em J. Approx. Theory}, 156(1):82--94, 2009.

\bibitem[DiSj99]{DiSj99} M.\ Dimassi and J.\ Sj\"ostrand, {\it Spectral
asymptotics in the semi-classical limit}, Cambridge University Press
1999.

\bibitem[EmTr05]{TrEm05}
M.~Embree and L.~N. Trefethen.
\newblock {\em {Spectra and Pseudospectra: The Behavior of Nonnormal Matrices
  and Operators}}.
\newblock Princeton University Press, 2005.

\bibitem[GuWoZe14]{GuMaZe14}
A.~Guionnet, P.~Matchett Wood, and O.~Zeitouni.
\newblock {Convergence of the spectral measure of non-normal matrices}.
\newblock {\em Proc.~AMS}, 142(2):667--679, 2014.

\bibitem[HaSj08]{HaSj08}
M.~Hager and J.~Sj{\"o}strand.
\newblock {Eigenvalue asymptotics for randomly perturbed non-selfadjoint
  operators}.
\newblock {\em Mathematische Annalen}, 342:177--243, 2008.

\bibitem[Ka97]{Kal97}
O.~Kallenberg.
\newblock {\em Foundations of Modern Probability}.
\newblock {Probability and its Applications}. Springer, 1997.

\bibitem[Sj10]{Sj09b}
J.~Sj{\"o}strand.
\newblock Counting zeros of holomorphic functions of exponential growth.
\newblock {\em Journal of pseudodifferential operators and applications},
  1(1):75--100, 2010.

\bibitem[Sj19]{Sj19}
J.~Sj{\"o}strand.
\newblock {\em Non-Self-Adjoint Differential Operators, Spectral Asymptotics
  and Random Perturbations}, Vol.~14 of {\em Pseudo-Differential Operators
  Theory and Applications}.
\newblock Birkh{\"a}user Basel, 2019.

\bibitem[SjVo16]{SjVo15b}
J.~Sj{\"o}strand and M.~Vogel.
\newblock Large bi-diagonal matrices and random perturbations.
\newblock {\em J. of Spectral Theory}, 6(4):977--1020, 2016.

\bibitem[SjVo19]{SjVo19} J.\ Sj\"ostrand, M.\ Vogel, {\it Toeplitz band matrices
    with small random perturbations,} \url{https://arxiv.org/abs/1901.08982}

\bibitem[SjZw07]{SjZw07}
J.~Sj{\"o}strand and M.~Zworski.
\newblock {Elementary linear algebra for advanced spectral problems}.
\newblock {\em Annales de l'Institute Fourier}, 57:2095--2141, 2007.

\bibitem[Ta12]{Ta12}
T.~Tao.
\newblock {\em {Topics in Random Matrix Theory}}, volume 132 of {\em {Graduate
  Studies in Mathematics}}.
\newblock American Mathematical Society, 2012.

\bibitem[TaVuKr10]{TVK10}
T.~Tao, V.~Vu, and M.~Krishnapur.
\newblock Random matrices: universality of esds and the circular law.
\newblock {\em The Annals of Probability}, 38(5):2023--2065, 2010.

\bibitem[Vo16]{Vo14}
M.~Vogel.
\newblock {The precise shape of the eigenvalue intensity for a class of
  non-selfadjoint operators under random perturbations}.
\newblock {\em to appear in Annales Henri Poincar{\'e}}, 2016.
\newblock e-preprint [arXiv:1401.8134].

\bibitem[Wo16]{Wo16}
P.~M. Wood.
\newblock Universality of the esd for a fixed matrix plus small random noise: A
  stability approach.
\newblock {\em Annales de l'Institute Henri Poincare, Probabilit{\'e}s et
  Statistiques}, 52(4):1877--1896, 2016.

\end{thebibliography}
\end{document}